\documentclass[]{amsart}

\usepackage[USenglish]{babel}
\usepackage[utf8]{inputenc}
\usepackage[T1]{fontenc}

\usepackage{amsfonts}
\usepackage{amssymb}
\usepackage{amsthm}
\usepackage{amscd}
\usepackage{amsxtra}
\usepackage{xy}
\usepackage{todonotes}
\usepackage{comment}

\usepackage{enumerate}
\usepackage{dsfont}
\usepackage{diagrams}
\usepackage{epic}
\usepackage{curves}
\usepackage[dvips]{rotating}

\usepackage{todonotes}
\usepackage[pdftitle={Lyapunov Exponents of Rank $2$-Variations of Hodge Structures and Modular Embeddings},
	    pdfauthor={Andre Kappes},
	    linktocpage=true
	    ]{hyperref}

\DeclareMathOperator{\Moduli}{\mathcal{M}}       %Modulraum
\DeclareMathOperator{\Teich}{\mathcal{T}}        %Teichmüllerraum
\DeclareMathOperator{\Stratum}{\Omega \mathcal{M}}      %Stratum gewisser hol. Differentiale
\DeclareMathOperator{\XFamily}{\mathcal{X}}      %Familie von Kurven
\DeclareMathOperator{\XFamilyuniv}{\mathcal{X}_{\textrm{univ}}}      %universelle Familie von Kurven
\DeclareMathOperator{\HH}{\mathbb{H}}            %upper half space
\DeclareMathOperator{\PP}{\mathbb{P}}            %projective space
\DeclareMathOperator{\DD}{\mathbb{D}}            %poincare disk
\DeclareMathOperator{\Graph}{\mathcal{G}}        %Graph

\DeclareMathOperator{\Odd}{\mathrm{odd}}         %odd index
\newcommand{\univ}{\mathrm{univ}}		       %universal	

\DeclareMathOperator{\Holomorphic}{\mathcal{O}}  %holomorphe Funktionen
\DeclareMathOperator{\ULocal}{\mathds{U}}
\DeclareMathOperator{\LLocal}{\mathds{L}}
\DeclareMathOperator{\VLocal}{\mathbb{V}}
\DeclareMathOperator{\WLocal}{\mathbb{W}}

\DeclareMathOperator{\VBundle}{\mathcal{V}}       %Vector bundle
\DeclareMathOperator{\UBundle}{\mathcal{U}}       %Vector bundle

\DeclareMathOperator{\Mod27}{\mathbf{N}}
\DeclareMathOperator{\Qmod9}{\mathbf{M}}
\DeclareMathOperator{\L22}{\mathbf{L}_{2,2}}
\DeclareMathOperator{\CharL22}{\widetilde{\mathbf{St}_3}}

\DeclareMathOperator{\Ori}{\mathbf{O}}

\DeclareMathOperator{\K}{\mathbb{K}}		  %complex numbers
\DeclareMathOperator{\CC}{\mathbb{C}}		  %complex numbers
\DeclareMathOperator{\RR}{\mathbb{R}}		  %reals
\DeclareMathOperator{\NN}{\mathbb{N}}		  %natural numbers
\DeclareMathOperator{\ZZ}{\mathbb{Z}}		  %integers
\DeclareMathOperator{\QQ}{\mathbb{Q}}		  %rationals
\DeclareMathOperator{\ra}{\rightarrow}           %Pfeil

\DeclareMathOperator{\GL}{\mathrm{GL}}           %General linear group
\DeclareMathOperator{\Out}{\mathrm{Out}}         %outer Automorphismengruppe
\DeclareMathOperator{\Aut}{\mathrm{Aut}}         %Automorphismengruppe
\DeclareMathOperator{\PSL}{\mathrm{PSL}}         %Projektive spezielle lineare Gruppe
\DeclareMathOperator{\SL}{\mathrm{SL}}           %Spezielle lineare Gruppe
\DeclareMathOperator{\Sp}{\mathrm{Sp}}           %Symplektische Gruppe
\DeclareMathOperator{\SO}{\mathrm{SO}}           %Spezielle orthogonale Gruppe
\DeclareMathOperator{\Stab}{\mathrm{Stab}}       %Stabilisator einer Menge
\DeclareMathOperator{\Gal}{\mathrm{Gal}}         %Galoisgruppe
\DeclareMathOperator{\MCG}{\mathrm{MCG}}         %Teichmller modular group
\DeclareMathOperator{\Aff}{\mathrm{Aff}}         %Affine Diffeomorphismen
\DeclareMathOperator{\Trans}{\mathrm{Trans}}     %Gruppe der Translationen
\newcommand{\dd}{\mathrm{d}}			 	%dx
\DeclareMathOperator{\der}{D}         		 	%Derivation eines affinen Diffeos
\DeclareMathOperator{\id}{\mathrm{id}}           %Identität

\renewcommand{\Re}{\mathrm{Re\,}}        
\renewcommand{\Im}{\mathrm{Im\,}}       %Imaginärteil

\DeclareMathOperator{\divisor}{\mathrm{div}}            %divisor
\DeclareMathOperator{\vol}{\mathrm{vol}}        %Volume
\DeclareMathOperator{\rank}{\mathrm{rk}}         %Rang
\DeclareMathOperator{\dom}{\mathrm{dom}}           %domain
\DeclareMathOperator{\range}{\mathrm{im}}         %Range
\DeclareMathOperator{\isom}{\cong}               %isomorph
\DeclareMathOperator{\subgp}{\leqslant}          %subgroup
\DeclareMathOperator{\tensor}{\otimes}           %Tensor

\DeclareMathOperator{\hyp}{\mathrm{hyp}}         %hyperbolic
\DeclareMathOperator{\Ad}{\mathrm{Ad}}           %Adjoint
\DeclareMathOperator{\Kern}{\mathrm{Ker}}        %Kernel
\DeclareMathOperator{\Comm}{\mathrm{Comm}}       %Commensurator
\DeclareMathOperator{\Image}{\mathrm{Im}}        %Image
\DeclareMathOperator{\ipair}{\mathit{i}}         %intersection pairing
\DeclareMathOperator{\per}{\mathit{p}}           %period map
\DeclareMathOperator{\Per}{\mathrm{Per}}         %period domain
\newcommand{\ol}[1]{\overline{#1}}	  	 %overline
\renewcommand{\mod}{\backslash}			 %Left action quotient

\newcommand{\set}[2]{\bigl\{#1\mid #2\bigr\}} 		%Menge der Form {#1 | #2}
\newcommand{\gen}[1]{\langle#1\rangle}			%thing generated by #1
\newcommand{\eg}{e.\,g.\ }		%i.e. Abkuerzung
\newcommand{\ie}{i.\,e.\ }		%i.e. Abkuerzung
\newcommand{\wrt}{w.\,r.\,t.\ }		%with respect to Abkuerznung
\newcommand{\textmatrix}[4]{\left(\begin{smallmatrix} #1&#2\\ #3&#4\\ \end{smallmatrix}\right)}

\newcommand{\OriSquare}[5]
{
\begin{xy}
(0,0)="Pos";
"Pos"+(1,0) **@{-}; ?*h!U!/^2pt/{\text{\scriptsize #5}}, 
"Pos"+(1,1) **@{-}; ?*h!L!/^1pt/{\text{\scriptsize #2}},
"Pos"+(0,1) **@{-}; ?*h!D!/^1pt/{\text{\scriptsize #3}},
"Pos" **@{-}; ?*h!R!/^1pt/{\text{\scriptsize #4}},
"Pos"+(0.5,0.5) *h{#1};
\end{xy}
}

\excludecomment{draftcomment}

\theoremstyle{plain}
\newtheorem{prop}{Proposition}[section]
\newtheorem{thm}[prop]{Theorem}
\newtheorem{lem}[prop]{Lemma}

\theoremstyle{definition}
\newtheorem{defn}[prop]{Definition}

\theoremstyle{remark}
\newtheorem{rem}[prop]{Remark}
\newtheorem{exa}[prop]{Example}
\numberwithin{equation}{section}

\title[Lyapunov Exponents of rank $2$-VHS]{Lyapunov Exponents of Rank $2$-Variations of Hodge Structures and Modular Embeddings}

\author[A. Kappes]{Andr\'e Kappes}
\address{Institut f\"ur Mathematik, Robert-Mayer-Str. 6--8,
 Goethe-Universit\"at Frankfurt/Main}
\email{kappes@math.uni-frankfurt.de}
\date{\today}
\thanks{The author is partially supported by ERC-StG 257137}
\begin{document}

\begin{abstract}
If the monodromy representation of a VHS over a hyperbolic curve stabilizes a rank two subspace, there is a single non-negative Lyapunov exponent associated with it. We derive an explicit formula using only the representation in the case when the monodromy is discrete.
\end{abstract}

\maketitle

\section{Introduction}
The Lyapunov exponents of a dynamical cocycle are usually hard to come by. The action of the Teichm\"uller geodesic flow on the relative cohomology bundle over the moduli space of curves $\Moduli_g$ is a striking exception, since much information can be obtained from a formula for the \textit{sum of the non-negative Lyapunov exponents} originally discovered by Kontsevich and Zorich \cite{kontsevichzorich} (see also \cite{forni02}, \cite{EKZbig}). It exploits a link between algebraic geometry and dynamical systems and expresses the sum as integrals over certain characteristic classes of $\Moduli_g$.
\par
Variants of this result are known to hold for subsets invariant under the flow such as Teichm\"uller curves, which are algebraic curves in $\Moduli_g$ isometrically embedded with respect to the Teichm\"uller metric. One can even replace the Teichm\"uller flow by the geodesic flow on an arbitrary hyperbolic curve $\HH/\Gamma$ (or more generally a ball quotient, see \cite{KapMoe12}) and an analogous formula will hold for the dynamical cocycle coming from the monodromy action of the fundamental group $\Gamma$ on the cohomology of a family of curves $\phi:\XFamily\to \HH/\Gamma$ (or more generally on a variation of Hodge structures (VHS) of weight one).
\par
In this paper, we focus on the situation when there exists a subbundle of rank two of such a relative cohomology bundle over a curve. It has only one non-negative Lyapunov exponent. Starting from the Kontsevich-Zorich formula, we show how to effectively compute this exponent only from the representation of the fundamental group.
\par
\begin{thm}\label{main thm}
 Let $\phi:\XFamily\to C$ be a family of curves over a non-compact algebraic curve $C=\HH/\Gamma$, and suppose there exists a rank $2$-submodule $V\subseteq H^1(\XFamily_c,\RR)$ invariant under the monodromy action $\rho_V$ of $\Gamma = \pi_1(C,c)$ such that $\rho_V(\Gamma)$ is a discrete subgroup of $\SL_2(\RR)$. 
\par
Then the non-negative Lyapunov exponent associated with $V$ is $0$ if $\Gamma$ acts as a finite group and is otherwise given by
\[\lambda = \frac{\vol(\HH/\rho_V(\Gamma))}{\vol(\HH/\Gamma)}\sum_{\Gamma_0\subgp\Gamma} (\Delta_0:\rho_V(\Gamma_0))\]
where $\Delta_0$ is a fixed parabolic subgroup of $\rho_V(\Gamma)$ and the sum runs over a system of representatives of conjugacy classes of maximal parabolic subgroups $\Gamma_0$ of $\Gamma$, whose generator is mapped to $\Delta_0 \setminus \{\pm I\}$.
\end{thm}
\par
If the relative cohomology of the family of curves over a (finite cover of a) Teichm\"uller curve has a rank-$2$ subbundle invariant under the flow and defined over $\QQ$, we can compute the associated Lyapunov exponent from the monodromy representation of the affine group. 
We carry this out for an example, where even a complete splitting into $2$-dimensional pieces is found.
\par
\begin{prop}\label{thm: example in genus 4}
 The Lyapunov spectrum of the Teichmüller curve generated by the square-tiled surface $(X,\omega)\in \Stratum_4(2,2,2)^{\Odd}$ given by
\[r = (1,4,7)(2,3,5,6,8,9)\quad \text{and}\quad u = (1,6,8,7,3,2)(4,9,5),\]
(see Figure \ref{fig:origami M}) is $1,\ \tfrac{1}{3},\ \tfrac{1}{3},\ \tfrac{1}{3},-\tfrac{1}{3},-\tfrac{1}{3},-\tfrac{1}{3},-1$.
\end{prop}
\par
Besides the Kontsevich-Zorich formula, the proof of Theorem~\ref{main thm} makes use of the period map $p:\HH\to \HH$ from the universal covering of $C$ to the classifying space of Hodge structures of weight one on a two-dimensional $\RR$-vector space. This map is equivariant for the two actions of $\Gamma$ and, in case $\rho_V(\Gamma)\subseteq \SL_2(\RR)$ is discrete and not finite, descends to a holomorphic map $\ol{p}$ between algebraic curves. The main observation is that the line bundle is a pullback of the cotangent bundle by $\ol{p}$, and that one can compute the degree of $\ol{p}$ by looking at the cusps. 
\par
From an abstract point of view, Theorem~\ref{main thm} deals with pairs $(p,\rho)$ of a homomorphism $\rho:\Gamma \to \SL_2(\RR)$ from a cofinite Fuchsian group $\Gamma$ and a holomorphic map $p:\HH\to\HH$ equivariant for the actions of $\Gamma$ and $\rho(\Gamma)$, which we call \textit{modular embeddings}. We show that these are rigid in the sense that $p$ and $\rho$ almost uniquely determine each other, a fact that has been remarked in \cite{mcmullenbild} for Teichm\"uller curves in genus $2$. Moreover, we introduce the notion of \textit{(weak) commensurability} of two modular embeddings (they must agree (up to conjugation) on some finite index subgroup). It follows that the Lyapunov exponent of a modular embedding is a weak commensurability invariant. We also investigate the commensurator of a modular embedding and show that it contains $\Gamma$ as a subgroup of finite index if $\rho$ has a non-trivial kernel.
\par
Every rational number in $[0,1]$ is a Lyapunov exponent of a Teichm\"uller curve in $\Moduli_g$ as can be deduced e.g. from \cite[Theorem 4.5]{bouwmoel}, \cite[Prop. 2]{cyclicekz} or \cite[Theorem 1.3]{Wright12-1}. However the denominator of the rational numbers that can be reached depends on $g$. In Proposition~\ref{prop: every ratl is lyap}, we combine the discussion of modular embeddings with Theorem~\ref{main thm} to obtain the same result by pulling back the universal family of elliptic curves via a complicated map. The resulting family will of course not map to a Teichm\"uller curve in moduli space.
\par
\subsection*{References} Previously, period maps have been used to compute the individual Lyapunov exponents of Teichm\"uller curves coming from abelian covers of $\PP^1$ \cite{Wright12-1}. In this situation, the period maps are Schwarz triangle maps, the monodromy is a possibly indiscrete triangle group, and the Lyapunov exponents are quotients of areas of hyperbolic triangles. Other examples, where individual Lyapunov exponents have been obtained by computing the degrees of line bundles, are the Veech-Ward-Bouw-M\"oller-Teichm\"uller curves \cite{bouwmoel}, \cite{Wright12-2}, cyclic covers of $\PP^1$ \cite{cyclicekz} and more generally Deligne-Mostow ball quotients \cite{KapMoe12}.
\par
Modular embeddings of $\HH$ into a product $\HH^k$ have been studied \eg in \cite{CoWo90} 
for the action of a Schwarz triangle group on the left and the direct product of 
its Galois conjugates on the right (where $k$ is the degree of the trace field over $\QQ$) or 
for non-arithmetic Teichm\"uller curves in \cite{moeller06}, \cite{mcmullenbild} for the action of the Veech group and its Galois conjugates.
%  or
% in \cite{HirzebruchZagier} for $k=2$ as modular curves embedded in Hilbert modular surfaces.
Our definition relates to theirs (for $k=2$) if 
one considers the $(\id,\rho)$-equivariant embedding $\HH\to \HH\times \HH$, $z\mapsto (z,p(z))$.
\par
\subsection*{Structure of the paper} The paper is organized as follows. Section \ref{sec:background} contains the necessary background on Teichm\"uller curves, variations of Hodge structures and Lyapunov exponents. Section \ref{sec:proofs} contains the proof of Theorem~\ref{main thm}. In Section \ref{sec:examples}, we discuss an algorithmic approach to the computation of Lyapunov exponents and present two examples, the one stated above being among them. Finally, in Section \ref{sec:period data}, we discuss various properties of modular embeddings.
\par
\subsection*{Acknowledgements} 
This paper grew out of the author's Ph.D. thesis \cite{KappesThesis}. The author thanks his advisors Gabi Weitze-Schmith\"usen, Frank Herrlich and Martin M\"oller for their support and the helpful discussions that led to his thesis and this paper. He also thanks Alex Wright for his comments on an earlier version of this paper.
\par

%%%%%%%%%%%%%%%%%%%%%%%%%%%%%%
\section{Background} \label{sec:background}
%%%%%%%%%%%%%%%%%%%%%%%%%%%%%%
In this section, we recall the concept of a variation of Hodge structures, the definition of the period map and the Kontsevich-Zorich formula and then specialize to the case of Teichm\"uller curves.
\par
\subsection{Variations of Hodge structures}
Let $C$ be a smooth algebraic curve over $\CC$, embedded in a projective 
curve $\ol{C}$. A family $\phi:\XFamily \to C$ of smooth curves defines a $\ZZ$-local
system $\VLocal = R^1\phi_*\ZZ$ on $C$, whose associated holomorphic vector bundle comes with a
 holomorphic subbundle $\VLocal^{1,0}\subset \VLocal\tensor_{\ZZ}\Holomorphic_C$, inducing
the Hodge decomposition of the cohomology in each fiber $\XFamily_c= \phi^{-1}(c)$.
This object, which is
actually the family of Jacobians associated with $\phi$, has been abstractly studied under the 
name \textit{variation of Hodge structures of weight $1$}; these consist of a $K$-local system 
$\VLocal$ on $C$ ($K$ a noetherian subring of $\RR$) and
a holomorphic subbundle $\VLocal^{1,0}\subset\VLocal\tensor_{K}\Holomorphic_C$ inducing a 
Hodge structure in each fiber.
% Note that in the case $K=\CC$, we no longer require symmetry of the Hodge
% decomp.
\par
Important for the study of variations of Hodge structures
 is the presence of a \textit{polarization}, which in our case is the 
intersection pairing on (co-)homology. It is defined as a locally constant alternating bilinear 
form $Q:\VLocal \tensor \VLocal \to K$ such that its $\CC$-linear extension satisfies
 the Riemann bilinear relations $Q(\VLocal^{1,0},\VLocal^{1,0}) = 0$ and $iQ(v,\ol{v}) > 0$ for
 non-zero $v\in \VLocal^{1,0}$. The norm $\|\cdot\|$ associated with the positive definite hermitian 
form $\tfrac{i}{2}Q(v,\ol{w})$ on $\VLocal^{1,0}$ on $\VLocal_{\RR}$ by 
\[\|v\| = \tfrac{i}{2}Q(v^{1,0},\ol{v^{1,0}})\]
(where $v^{1,0}$ denotes the projection of $v\in \VLocal_ {\RR}$ to $\VLocal^{1,0}$) is called Hodge norm. In the following, we write VHS for ``polarized variation of Hodge structures of weight 1''.
\par
By a \textit{local monodromy} of $\VLocal$ about a puncture $c\in \ol{C}\setminus C$, we shall understand the action 
of a small loop about $c$ on the fiber $\VLocal_c$ of a nearby point $c$. If $K$ is a number field, then by a Theorem of Borel, these 
transformations are always quasi-unipotent. 
If they are unipotent, then there is a canonical extension
 due to Deligne of $\VLocal\tensor_{K}\Holomorphic_C$ to a vector bundle $\VBundle$ on $\ol{C}$. The 
extension of the $(1,0)$-part inside $\VBundle$ will be denoted $\VBundle^{1,0}$.
\par
The (global) \textit{monodromy} is the linear representation of $\pi_1(C,c)$ on $\VLocal_c$ associated with the local
system $\VLocal$ (and uniquely determined up to conjugation).
\par
A standard reference for variations of Hodge structures is \cite{CMSP}.
% the notation $\VLocal$ will mean a polarized VHS of
% weight $1$ on a
% local system of finitely generated free $\ZZ$-modules, whereas for a field
% $K$, usually taken to be $\QQ$, $\RR$ or $\CC$, we denote the
% associated polarized $K$-VHS by $\VLocal_K = \VLocal\tensor_{\ZZ}K$.

\subsubsection{Decomposition of a VHS}
% Our goal, among others, is to determine for a given family of curves a decomposition of the associated VHS. 
% There are two important results for polarized VHS that we will make use of: 
By the work of Deligne, the category of $\CC$-VHS on a
quasiprojective base is semisimple. More precisely \cite{delfinitude},
 if $\VLocal$ is a VHS on a smooth quasiprojective algebraic variety $X$ over $\CC$, then
\begin{align}\label{eqn:Deligne-decomp}
\VLocal \isom \bigoplus_i \VLocal_i\tensor W_i 
\end{align}
where $\VLocal_i$ are irreducible local systems, $\VLocal_i\ncong \VLocal_j$ for
$i\neq j$ and $W_i$ are complex vector spaces. Moreover, each $\VLocal_i$ carries a 
 VHS unique up to shifting of the bigrading, such that 
\eqref{eqn:Deligne-decomp} is an isomorphism of VHS.
% \par
% Secondly, the category
%  of VHS on such a base satisfies a certain rigidity property: any morphism $\VLocal\to \WLocal$ of local 
%  systems carrying a VHS, which at one point is a 
% morphism of Hodge structures, is already a morphism of VHS.

\subsubsection{The period map and the period domain}
Let $x\in C$ be a base point and let $\VLocal$ be a VHS on $C$. The underlying local system
corresponds to the monodromy action of $\pi_1(C,x)$ on the fiber $\VLocal_x$ by continuation of
local sections along paths. The distinguished subspace $\VLocal^{1,0}_x$ of the Hodge filtration 
will be moved by this action; this movement is recorded by the \textit{period map} $\per:\tilde C\to \Per(\VLocal_x)$, which is 
a holomorphic map from the universal cover $u:\tilde C\to C$ to the \textit{period domain} $\Per(\VLocal_x)$, the classifying space of polarized 
Hodge structures that can be put on $\VLocal_x$. 
\par
The period map can be described in the following way: 
On $\tilde C$, the local system can be
globally trivialized by the constant sheaf $V$ of fiber $\VLocal_x$
 and the inclusion $u^*\VLocal^{1,0}\to V$ yields for every point $z\in \tilde C$
a Hodge structure on $V_z\isom \VLocal_x$, thus a point $\per(z)\in \Per(\VLocal_x)$.
%  A polarized VHS $\VLocal$ on $C$ defines a
% holomorphic map $\per: \tilde C \to \Per(\VLocal_x)$ from the universal 
% cover of $C$ to the period domain, the classifying space of polarized Hodge
% structures of a given type that can be put on $\VLocal_x\tensor\RR$. 
The fact that $u^*\VLocal^{1,0} \to V$ is an inclusion of sheaves with $\pi_1(C,x)$-action, 
corresponds to the map $\per$ being equivariant with respect to the action 
of $\pi_1(C,x)$ on $\tilde C$
by deck transformations and on $\Per(\VLocal_x)$ by the monodromy action.
\par
In the case of an $\RR$-VHS of weight $1$
and rank $2k$, $\Per(\VLocal_x) \isom \HH_k$, the Siegel upper halfspace
of dimension $k$ and the monodromy is a
representation of $\pi_1(C,x)$ into $\Sp_{2k}(\RR)$.
\par
A VHS $\VLocal$ on a curve $C$ is called \textit{uniformizing} if its period
map is biholomorphic. In this case, 
$(\VBundle^{1,0})^{\tensor 2} \isom \Omega^1_{\ol{C}}(\log S)$
where $S = \ol{C}\setminus C$ is the finite set of cusps. This isomorphism is given by the Kodaira-Spencer map, the
 only graded piece of the Gau\ss-Manin connection.
\par
In particular, there is a tautological uniformizing VHS on each period domain and each VHS 
is equal to the pullback of a tautological VHS on its period domain via the period map. 
We sketch this for a rank $2$-VHS, \ie $k=1$.
\par
Suppose we are given a holomorphic map 
$p:\tilde C \to \HH$ from the universal cover of a curve $C$, together
with a group homomorphism $\rho:\pi_1(C,x) \to \Sp_{2}(\RR) = \SL_2(\RR)$. 
The trivial bundle $\tilde C\times \RR^2 \to \tilde C$ acquires a $\pi_1$-action by
\[(z,v) \mapsto (\gamma(z), \rho(\gamma)(v)),\quad \gamma\in\pi_1(C,x),\quad \rho(\gamma) = \textmatrix{a}{b}{c}{d}\]
and hence gives rise to an $\RR$-local system $\VLocal$ on $C$ since the transition matrices are constant. In the same way, the trivial line bundle $\tilde C\times \CC \to \tilde C$ is acted upon by $\pi_1(C,x)$ by
\[(z,\lambda) \mapsto (\gamma(z), (cz+d)^{-1}\lambda)\]
and the inclusion
\[\tilde C\times \CC \to \tilde C\times \CC^2,\quad (z,\lambda) \mapsto (z, \lambda(p(z),1)^T)\]
is $\pi_1$-equivariant and hence descends to an inclusion of vector bundles $\VLocal^{1,0} \to \VLocal \tensor_{\RR} \Holomorphic_C$ on $C$.
Since $\per(\HH)\subseteq \HH$, the standard symplectic form on $\RR^2$ with matrix $\textmatrix{0}{1}{-1}{0}$ furnishes a polarization of this VHS. Moreover, if $\range(\rho)\subseteq \SL_2(\ZZ)$, then the lattice $\tilde C\times \ZZ^2 \subset \tilde C\times \RR^2$ is preserved and descends to a $\ZZ$-local system $\VLocal_{\ZZ}$ on $C$. We put $\VLocal^{0,1} = \VLocal\tensor_{\RR}\Holomorphic_C/\VLocal^{1,0}$. The quotient of $\VLocal^{0,1}$ by the image of $\VLocal_{\ZZ}$ is then a family of elliptic curves.
\par
\subsection{Teichm\"uller curves}
%A Teichm\"uller curve is an algebraic curve $C$ in the moduli space of curves 
%$\Moduli_g$, which is totally geodesic with respect to the Teichm\"uller metric.
We recall the basic definitions for Teichm\"uller curves and show that they fit into 
the above abstract setting with the slight modification that we have to deal with orbifold 
fundamental groups. Good surveys on this subject
 are e.g. \cite{mcmullenbild}, \cite{moelPCMI}, \cite{HubertSchmidt06} or \cite{HerSch06}.
\par
It is well-known that every Teichm\"uller curve in $\Moduli_g$ arises as the composition of a
Teichm\"uller embedding $j: \HH \to \Teich_g$ with the natural projection
$\Teich_g\to \Moduli_g$, and that a Teichm\"uller embedding is in turn determined by
a pair $(X,q)$ of a compact Riemann surface $X$ with a non-zero quadratic
differential $q$. Using a canonical double covering construction one can confine oneself
to $q= \omega^2$, where $\omega$ is a holomorphic $1$-form. Then the natural atlas on $X\setminus \divisor(\omega)$
obtained by locally integrating $\omega$ has only translations as transition maps, 
and we call the pair $(X,\omega)$ a \textit{translations surface}.
\par
Let $\Stratum_g$ be the moduli space of translation surfaces. It is stratified by the number of zeros of $\omega$. For a partition $(\kappa_1,\dots,\kappa_r)$ of $2g-2$, let 
 $\Stratum_g(\kappa_1,\dots,\kappa_r)$ denote the moduli space of translation surfaces $(X,\omega)$, 
where $\omega$ has $r$ zeros with multiplicities $\kappa_1,\dots,\kappa_r$.
\par
A homeomorphism $f:X\to X$ is called \textit{affine} if it acts as an affine linear map in the charts of
the translation structure. This is the case if and only if
its action on $H^1(X,\RR)$ preserves the subspace spanned by $\Re\omega$, $\Im\omega$. 
The group of all orientation-preserving affine homeomorphisms is denoted
by $\Aff(X,\omega)$. 
\par
Taking the derivative of an affine map induces a group homomorphism 
\[\der : \Aff(X,\omega)\ra \SL_2(\RR),\]
whose image is called the \textit{Veech group} $\SL(X,\omega)$ and whose kernel is the \textit{group of translations} $\Trans(X,\omega)$. 
The Veech group is a nonuniform discrete subgroup of $\SL_2(\RR)$ and a lattice if and only if the Teichm\"uller embedding associated with $(X,\omega)$ leads to a Teichm\"uller curve. In this case, we call the associated surface $(X,\omega)$ a 
\textit{Veech surface} and say that the Teichm\"uller curve is generated by $(X,\omega)$.
\par
The affine group acts naturally as a subgroup of the mapping class group $\MCG_g$ on the
Teichm\"uller disk, respectively as a group of orientation preserving isometries 
on $\HH =
\SO(2)\mod\SL_2(\RR)$ by the representation $\der$, and the Teichm\"uller embedding is equivariant for these
two actions. This action need not be free, but the kernel $\Aut(X,\omega)$ of affine biholomorphisms of $X$ is always finite. If $(X,\omega)$ generates the Teichm\"uller curve $C$, the curve $\HH/\Aff(X,\omega)$ is the normalization of $C$ and $\Aff(X,\omega)$ is the orbifold fundamental group.
In particular, if we view the inclusion $\ol{\jmath}: \HH/\Aff(X,\omega) \to \Moduli_g = \Teich_g/\MCG_g$ as an inclusion 
of orbifolds or stacks, we can pull back the universal family over $\Moduli_g$ to obtain a canonical family of curves over 
a Teichmüller curve. However, to avoid the notion of stacks, we always pass to a suitable finite 
index subgroup $\Gamma\subgp \Aff(X,\omega)$, where a map to a fine moduli space
and thus a family $\phi:\XFamily\to \HH/\Gamma$ exists (see \cite[1.4]{moeller06} for details).
\par
% not yet needed
% We shall also use the composition $\Pder$ of $\der$ with the projection $\SL_2(\RR)\to \PSL_2(\RR)$.
%  Note that $f\in \Aff(X,\omega)$ is holomorphic iff $\der(f) = \{\pm I\}$, \ie $f\in \Kern(\Pder)$.
%  The image of $\Pder$ is called the Veech group $\Gamma(X,\omega)$.
% vielleicht spaeter noch rein!
%  Note that this is not the
% usual action by M\"obius transformations. 
\subsubsection{Origamis}
An \textit{origami}, also called \textit{square-tiled surface} is a translation surface $\Ori = (X,\omega)$
together with a holomorphic map $p:\Ori \to E= \CC/\ZZ\oplus i\ZZ$ such that $p$ is ramified at most over one point $e\in E$, and such that $\omega = p^*\dd z$.
\par
Origamis give rise to Veech surfaces, since their Veech groups are commensurable with $\SL_2(\ZZ)$. If 
$\Ori$ is \textit{primitive}, \ie $p$ does not factor into $f\circ p'$ where $f:E'\to E$ is an isogeny between genus $1$-surfaces of degree $>1$, then $\SL_2(\ZZ)$ is a subgroup of finite index of $\SL_2(\ZZ)$. The same holds if we consider instead $\Ori^* = \Ori\setminus p^{-1}(e)$ and affine maps preserving $p^{-1}(e)$.
\par
An origami of degree $d$ is conveniently described by two permutations $r$, $u\in S_d$ that prescribe 
how $d$ unit squares are glued along their edges: we identify the right (respectively upper) edge of  square $i$ with the left (respectively lower) edge of square $r(i)$ (respectively $u(i)$). If the subgroup generated by $r$ and $u$ acts transitively, then the resulting topological space is connected and the tiling by squares defines a covering map to $E$, ramified at most over $\ol{0}\in E$.
\par
More on origamis can be found e.g. in \cite{gabialgo} or \cite{ZmiaikouThesis}.

\subsubsection{Monodromy representation}
The \textit{monodromy representation} of the orbifold fundamental group $\Aff(X,\omega)$ of a Teichm\"uller curve 
is the representation 
\[\rho:\Aff(X,\omega) \to \Sp(H^1(X,\ZZ),\ipair^*),\quad f\mapsto (f^{-1})^*\]
It respects the algebraic intersection pairing $\ipair^*$ on cohomology. 
One can show that $\rho$ is actually injective and that $\rho$, restricted to a suitable
finite index subgroup where the family $\phi:\XFamily\to \HH/\Gamma$ exists, 
is the monodromy representation associated with $R^1\phi_*\ZZ$ (see \cite{Bauer09} for the proof of both statements).
% To actually make a
% connection of $\rho$ with a monodromy representation of a family of curves, we
% need to pass to an appropriate finite cover of the Teichm\"uller curve,
% respectively $\HH/\Gamma(X,\omega)$. By \ref{} we can choose a finite index
% subgroup $\Gamma\subgp \Aff(X,\omega)$ such that the map $\HH/\Gamma \to
% \Moduli_g$ factors over a fine moduli space. 
\par
% equivariance for factors
In the case of a Teichm\"uller curve, the equivariance carries over to the
possibly non-free action of $\Aff(X,\omega)$ on $\HH$ and on $\HH_k$ via its
monodromy representation.
This is easily seen as follows. The Teichm\"uller
embedding $j:\HH\to \Teich_g = \Teich(X)$ associated with $(X,\omega)$ is
equivariant
with respect to the action of $f\in \Aff(X,\omega)$ by $\der(f)$ on $\HH$ and by
its action as element of the mapping class group, that sends the marked
Riemann 
 surface $(X_\tau,m_{\tau})$ to $(X_{\tau}, m_{\tau}\circ f^{-1})$. The
natural 
map $t: \Teich(X) \to \HH_g$ is in turn equivariant with respect to
the Torelli 
morphism $\MCG_g \to \Sp(2g,\ZZ)$, $f\mapsto (f^{-1})^*$. The
period map $p_{\phi_{\univ}}$ of the
 pullback family $\phi_{\univ}:\XFamily = \XFamilyuniv \times_j \HH \to \HH$ of the
universal family of curves $\XFamilyuniv \to \Teich_g$ is now given as the
composition of $t\circ j$.
\par
If $\Gamma$ is a finite-index subgroup preserving a subspace $W$ of $H^1(X,\RR)$, then
the associated representation will induce a sub-local system $\WLocal$ of
 $R^1\phi_*\RR$ 
on some $\HH/\Gamma'$ for a suitable finite index subgroup $\Gamma'\subgp \Gamma$. 
Applying Deligne's semisimplicity result,
we find that $\WLocal$ carries a VHS, and $R^1\phi_*\RR = \WLocal \oplus \tilde \WLocal$
where $\tilde \WLocal$ is the complement of $\WLocal$.
\begin{draftcomment}
\textit{needs some more justification concerning that
  splitting is preserved by $\Gamma$. Not really, is implicitly given!}
\end{draftcomment}
Therefore, we can find a trivialization of the pullback local system 
on $\HH$, \ie a basis
of $H_1(X,\RR)$ such that with respect to this basis, the period map $p_{\phi_{\univ}}$ is given as
\[z\mapsto  \begin{pmatrix}
                      Z_1(z) & 0 \\
		      0  & Z_2(z)
                     \end{pmatrix} \in \HH_g,\]
where $Z_1$ and $Z_2$ are square matrices of dimensions $\rank \WLocal^{1,0}$ and $\rank\tilde \WLocal^{1,0}$.
This map is equivariant for all $\gamma\in \Aff(X,\omega)$ such that $\rho(\gamma)$ respects the 
decomposition $W \oplus \tilde W$. In particular, the period map 
\[p_W : \HH \to \HH_{\rank W}, \quad z\mapsto Z_1(z)\]
associated with the VHS $\WLocal$ 
is $\Gamma$-equivariant (and not just $\Gamma'$-equivariant).

\subsubsection{The VHS of the family of curves over a Teichm\"uller curve}

Using Deligne's result, M\"oller characterizes the VHS on a Teichm\"uller curve
\cite{moeller06} generated by a translation surface $(X,\omega)$. After
passing to a finite cover, the VHS on a Teichm\"uller
curve always admits a uniformizing direct factor $\LLocal$ in its VHS,
 defined over the trace field of $\SL(X,\omega)$, 
whose local system is given by the Fuchsian
representation $\der$ of $\Aff(X,\omega)$. Conversely, he shows that if a family of
curves $\phi:\XFamily\to C$ over a curve $C$ has a uniformizing direct summand
$\LLocal$ in its $\RR$-VHS $R^1\phi_*\RR$, then $C$ is a finite cover of a
Teichm\"uller curve.
%
%Q: Does non-compactness enter here? Or is this a consequence of the proof}
%A: No, it is apparently a consequence of the proof. Do we have an independent
%   proof establishing noncompactness?
%

% 
%  More precisely, if $C = \HH/\Gamma$ with
% $\Gamma\subgp \Aff(X,\omega)$ an appropriate finite index subgroup, and if $c\in
% C$ is a point corresponding to $(X,\omega)$, then the stalk of $\LLocal_{\RR}$
% at $c$ is spanned by $\Re\omega$, $\Im\omega$ and the Hodge decomposition is
% $\LLocal_{\CC} = \CC\omega \oplus \CC\ol{\omega}$. With respect to this basis,
% $f\in \Gamma$ acts by $(D(f)^{-1})^T$. 
\par
\subsection{Lyapunov exponents}
Lyapunov exponents are characteristic numbers associated with certain dynamical
systems. In our case of a $\RR$-VHS $\VLocal$ on a hyperbolic curve $C =
\HH/\Gamma$, they measure the logarithmic growth rate of the Hodge norm of a
vector in $\VLocal_x$ when being dragged along a generic (\wrt the Haar measure) geodesic on
$\HH/\Gamma$ under parallel transport. 
\par
For an $\RR$-VHS of rank $2k$, the Lyapunov spectrum consists of $2k$
exponents, counted with multiplicity that group symmetrically around $0$
\[\lambda_1\geq \dots \geq \lambda_k \geq 0 \geq \lambda_{k+1} = -\lambda_{k}
\geq \dots \geq \lambda_{2k} = -\lambda_1.\]
One usually normalizes the curvature in order that $\lambda_1 =1$ 
($K = -4$ in the case of hyperbolic curves). In the case of a 
Teichm\"uller curve, we further have $\lambda_1 = 1 > \lambda_2$. In general,
virtually all knowledge about individual exponents stems from using variants of
a formula for the sum over the first half of the spectrum which we refer 
to as the \textit{non-negative Lyapunov spectrum} in the following. 
This formula is originally due to Kontsevich and Zorich \cite{kontsevichzorich},
and was rigorously proved in \cite{bouwmoel}, \cite{EKZbig} or
\cite{forni02}. A variant of it can be stated as follows.
%  (see \cite{KapMoe12} for a proof in this case).

\begin{thm} \label{thm:sumformula}
 Let $\VLocal$ be an $\RR$-VHS of weight $1$ and rank $2k$ on a (possibly non-compact) curve $C =
\HH/\Gamma$. Then the non-negative Lyapunov exponents
$\lambda_1,\dots,\lambda_k$ of $\VLocal$ satisfy
\begin{align} \label{eq:sum-lyap-formula}
\lambda_1+ \dots +\lambda_k =
\frac{2 \deg(\VBundle^{1,0}) }{2g(\ol{C}) -2 + s} 
\end{align}
where $\ol{C}$ is the completion of $C$, $s = |\ol{C}\setminus C|$, and
$\VBundle^{1,0}$ is the Deligne extension of $\VLocal^{1,0}$ to $\ol{C}$.
\end{thm}

A generalization of this formula to higher dimensional ball quotients also
exists \cite{KapMoe12}, as well as an explicit formula for the sum of Lyapunov
exponents of the relative cohomology in case the Teichm\"uller curve 
is generated by an origami
\cite{EKZbig}.

Using Theorem \ref{thm:sumformula}, individual Lyapunov exponents have been
computed \eg for families of cyclic and abelian coverings of $\PP^1$ ramified
over $4$ points (\cite{cyclicekz}, \cite{Wright12-1}), in genus two
\cite{Bainbridge07} and for all known primitive Teichm\"uller curves in higher
genus \cite{bouwmoel}.

We recall two important properties of the Lyapunov spectrum. First, it
remains unchanged if we pass to a finite index subgroup $\Gamma'$ and consider
the Lyapunov spectrum of the pullback VHS on $\HH/\Gamma'$ (see e.g. \cite[Proposition 5.6]{KapMoe12}). Secondly, if the
VHS splits up as a direct sum, then its Lyapunov spectrum is the union of the
spectra of its pieces. 

%**********************************
\section{Lyapunov exponents of rank 2-VHS} \label{sec:proofs}
%**********************************
In this section, we derive the main theorem from Theorem \ref{thm:sumformula}.
\par
\begin{prop} \label{prop:Rank2-formula}
Let $\rho:\Gamma\to \SL_2(\RR)$ be a group homomorphism such that $\Gamma$ and 
$\Delta = \rho(\Gamma)$ are cofinite, torsionfree Fuchsian
groups, and
let $p:\HH\to \HH$ be a non-constant $\rho$-equivariant holomorphic map. Let
$\ol{p}:\HH/\Gamma \to \HH/\Delta$ be the map induced by $p$, and let $\VLocal$
be the pullback by $\ol{p}$ of the universal rank-2 $\RR$-VHS on $\HH/\Delta$. 
Then the non-negative Lyapunov exponent of $\VLocal$ is given by
\begin{align} \label{eqn:Lyap-with vol}
\lambda = \frac{\deg(\ol{p})\vol(\HH/\Delta)}{\vol(\HH/\Gamma)}.
\end{align}
\end{prop}
\begin{proof}
 By Theorem \ref{thm:sumformula}, the Lyapunov exponent is given by
\[\lambda = \frac{2\deg(\VBundle^{1,0})}{\deg(\omega_{\ol{B}})},\]
where $\ol{B}$ is the completion of $\HH/\Gamma$ and where $\VBundle^{1,0}$ is
the Deligne extension to $\ol{B}$ of the $(1,0)$-part of $\VLocal$. Further,
\[\deg(\omega_{\ol{B}}) = -\chi(\ol{B}) = \tfrac{1}{2\pi}
4\vol(\HH/\Gamma),\]
by the Gau\ss-Bonnet formula (where we take the curvature on $\HH$ to be normalized to $-4$).
Let $\ol{C}$ be the completion of $\HH/\Delta$,
and let $\ULocal$ be the universal VHS on $\HH/\Delta$, whose Deligne extension
of the $(1,0)$-part we denote by $\UBundle^{1,0}$. By universality and
the Gau\ss-Bonnet formula,
 we have
\[2\deg(\UBundle^{1,0}) = \deg(\omega_{\ol{C}}) =
\tfrac{1}{2\pi}4\vol(\HH/\Delta),\]
and since $\ol{p}^*\UBundle^{1,0} = \VBundle^{1,0}$, the claim follows.
\end{proof}
\par
We remark that Proposition \ref{prop:Rank2-formula} is also readily deduced from a reformulation of the Kontsevich-Zorich formula by Wright \cite[Theorem 1.2]{Wright12-1}.
\par
For our applications, we need to allow groups that contain torsion elements or
whose action on $\HH$ has a (usually finite) kernel. In this situation there
might not be a VHS on the quotient, but only on an appropriate finite cover.
 (Note that by a theorem of Selberg, any finitely generated subgroup
of a matrix group always has a torsionfree subgroup of finite
index.)
However, we still can compute the right-hand side 
of \eqref{eqn:Lyap-with vol}. The next lemma shows that 
this quantity is independent under passing to a
finite index subgroup.
\par
\begin{lem} \label{lem:going up}
 Let $\Gamma$ be a group acting cofinitely and holomorphically on $\HH$. Let
$\rho:\Gamma\to \SL_2(\RR)$ be a group homomorphism such that $\Delta =
\rho(\Gamma)$ is a cofinite Fuchsian
group, and let $p:\HH\to \HH$ be a non-constant $\rho$-equivariant holomorphic
map. Let $\Gamma'\subgp \Gamma$ be a finite index subgroup. Then $\Delta'=
\rho(\Gamma')$ has finite index and
\[\frac{\deg(\ol{p})\vol(\HH/\Delta)}{\vol(\HH/\Gamma)} =
\frac{\deg(\ol{p}')\vol(\HH/\Delta')}{\vol(\HH/\Gamma')},\]
where $\ol{p}:\HH/\Gamma\to \HH/\Delta$ and $\ol{p}':\HH/\Gamma' \to
\HH/\Delta'$ are the maps induced by $p$.
\end{lem}
\begin{proof}
We have $(\Delta:\Delta') = (\Gamma:\rho^{-1}(\Delta')) \leq
(\Gamma:\Gamma')$. The second claim follows by comparing the degrees of maps in
the commutative diagram
\[\begin{CD}
 \HH/\Gamma' @>{\ol{p}'}>> \HH/\Delta'\\
 @VVV	@VVV \\
 \HH/\Gamma @>>{\ol{p}}> \HH/\Delta
\end{CD}\]
%  \begin{diagram}[h=2em,w=3em,nohug]
%   \HH/\Gamma' & \rTo^{\ol{p}'} & \HH/\Delta'\\
%   \dTo       &        & \dTo\\
%   \HH/\Gamma & \rTo_{\ol{p}} & \HH/\Delta
%  \end{diagram}
\end{proof}
\par

%**********************************
\subsection{Computing the degree  of $\overline{p}$}
%**********************************
In this section, we show that in the presence of cusps, the quantities on the
right hand side of \eqref{eqn:Lyap-with vol} are explicitly computable only from
the group homomorphism $\rho$.
\par
Throughout, let $\rho:\Gamma \to \Delta$ be a
homomorphism between non-cocompact, cofinite Fuchsian groups, let $p:\HH\to \HH$
be a $\rho$-equivariant non-constant holomorphic map, and let $\ol{p}:\HH/\Gamma
\to \HH/\Delta$ be the map induced by $p$. Denote the extension
$\ol{p}:\ol{B}\to \ol{C}$ to the completions $\ol{B}$ of $\HH/\Gamma$ and
$\ol{C}$ of $\HH/\Delta$ by the same letter.
\par
In the following, a cusp will, depending on the context, 
be a point in $\partial \HH$, stabilized by a parabolic in $\Gamma$ or its equivalence class under the
action of $\Gamma$, respectively the 
point in the completion of $\ol{B}$ corresponding to this class.
\par 
\begin{lem}\label{lem:compute-degree}
 Let $\Gamma_0 \subgp \Gamma$, respectively $\Delta_0\subgp \Delta$ be
maximal parabolic subgroups associated with cusps $b\in \ol{B}$, respectively $c\in
\ol{C}$. Let $\gamma$ be a generator of $\Gamma_0$ such that $\rho(\gamma)$ is
parabolic and lies in $\Delta_0$. Then
\begin{enumerate}[a)]
 \item $\ol{p}$ maps $b$ to $c$. 
 \item The ramification index $e(\ol{p},b)$ of $\ol{p}$ at $b$ is
$(\Delta_0:\rho(\Gamma_0))$.
 \item We have $\deg(\ol{p}) = \sum_{b\in \ol{p}^{-1}(c)} e(\ol{p},b)$.
\end{enumerate}
\end{lem}
\begin{proof}
 Let $s$, respectively $t \in \RR\cup\{\infty\}$ be the fixed point of $\Gamma_0$,
respectively $\Delta_0$. Without loss of generality, we may assume $s = t = \infty$, and that $\Gamma_0$
respectively $\Delta_0$ is generated by $(z\mapsto z+1)$. The canonical projections
$u_\Gamma: \HH \to \HH/\Gamma$ respectively $u_\Delta:\HH\to\HH/\Delta$ factor over
$\HH\to \HH/\Gamma_0$ respectively $\HH\to\HH/\Delta_0$, and both $\HH/\Gamma_0$ and
are $\HH/\Delta_0$ isomorphic to $\DD^*$ via the map induced by $z\mapsto
\exp(2\pi i z)$. Under this isomorphism, the image of $s$, respectively $t$
is identified with $0 \in \DD$. Being equivariant, the map $p$  descends to
$p_0:\HH/\Gamma_0\isom \DD^* \to \DD^*\isom \HH/\Delta_0 $. 
\par
To prove a), it suffices to show that for a
sequence in $\DD^*$ converging to $0$, the image under $p_0$ converges to $0$.
Define $a_n =
in$ and let $b_n = \exp(2\pi ia_n)$ in $\DD^*$; we have
$b_n\to 0$. By the Schwarz lemma, $p$ does not increase hyperbolic distances,
thus
\[d_{\hyp}(a_n,a_n+1) \geq d_{\hyp}(p(a_n),p(a_n)+\lambda),\]
where $z\mapsto z+\lambda$, ($\lambda\in \ZZ\setminus \{0\}$) generates
$\rho(\Gamma_0)$. Since
$d_{\hyp}(a_n,a_n+1) \to 0$ as $n\to\infty$, we also have
$d_{\hyp}(p(a_n),p(a_n)+\lambda) \to 0$, whence $\Im(p(a_n)) \to \infty$, which
means
that $p_0(b_n) \to 0$.
\par 
b) A basis of punctured neighborhoods of $b\in \ol{B}$ is given by the images of horoballs
$U_R = \set{z\in\HH}{\Im(z)>R}$ under the projection modulo $\Gamma$. If we choose 
$R$ big enough, then we can ensure that $U_R$ is stabilized only by elements 
of $\Gamma_0$, whence the quotient $U_R/\Gamma_0$ embeds into $\HH/\Gamma$, and gives rise to a 
chart $U_R/\Gamma_0\to \DD^*$. In the same way, we can obtain a chart $U_{R'}/\Delta_0\to \DD^*$
 such that in these charts, $\ol{p}$ takes the form $z\mapsto z^k$ with $k$ being the ramification
index. Thus the induced map $\ol{p}_*$ on fundamental
groups maps a generator of $\pi_1(U_R/\Gamma_0)\isom \Gamma_0$ to the $k$-th power of a generator of
$\pi_1(U_{R'}/\Delta_0) \isom \Delta_0$. This group homomorphism $\Gamma_0\to \Delta_0$ 
must be equal to $\rho$, since for both $p$ is equivariant. It follows that 
 $k = (\Delta_0:\rho(\Gamma_0))$.
\end{proof}
\par
Note that the degree of $\ol{p}$ can be $1$ without $p$ being an isomorphism. However, this can happen only
when the Fuchsian groups contain torsion elements.
\par
\begin{proof}[Proof of Theorem~\ref{main thm}]
The Lyapunov spectrum does not change, if we pass to a finite index subgroup $\Gamma'$ of $\Gamma$. Thus if $\rho_V(\Gamma)$ is finite, then $\rho_V(\Gamma')$ will be trivial for the finite index subgroup $\Gamma' = \Kern(\rho_V)$, and therefore $\lambda = 0$.
\par
We are left with the case when $\rho_V(\Gamma)$ is infinite. By Deligne's semisimplicity theorem, the local system $\VLocal$ associated with $V$ carries a VHS. We let $p_V$ be its period map. $p_V$ cannot be constant, for otherwise every $g\in \rho_V(\Gamma)$ would stabilize $p(z) \equiv \text{const} \in \HH$, but this stabilizer is finite since $\rho_V(\Gamma)$ is discrete. Thus we obtain a non-constant holomorphic map $p:\HH \to \HH$ that descends to $\ol{p}:\HH/\Gamma \to \HH/\rho_V(\Gamma)$. On the left-hand side, we have a Riemann surface of finite type. We claim that $\ol{p}$ can be extended continuously to the compactification $\ol{B}$ of $\HH/\Gamma$, respectively the possibly only partial compactification $\ol{C}$ of $\HH/\rho_V(\Gamma)$, where $\ol{B}$ and $\ol{C}$ are obtained by adjoining all cusps. From this we conclude that $\ol{C}$ is compact and thus $\HH/\rho_V(\Gamma)$ has finite volume.
\par
To prove the claim, let $b\in \partial \HH$ be a cusp of $\Gamma$ and let $\gamma$ be a generator of its stabilizer. By the Schwarz lemma, it follows that
\[d_{\HH}(z,\gamma(z)) \geq d_{\HH}(p(z),p(\gamma(z)) \geq \ell(\rho(\gamma))\]
where $\ell(g) = \inf_{z\in \HH} d_{\HH}(z,g(z))$ is the translation length. Since the left-hand side goes to $0$ as $z$ approaches the cusp, $\ell(\rho(\gamma)) =0$, whence $\rho(\gamma)$ is either parabolic or elliptic. In the first case, the proof of Lemma \ref{lem:compute-degree} a) shows that $\ol{p}$ is locally given as a holomorphic map $\DD^*\to\DD^*$, which has a canonical extension to $\DD\to \DD$. This is true also for the second case with the difference that $\ol{p}(b)$ is now a point in $\HH/\rho_V(\Gamma)$.
\par
The statement of Theorem~\ref{main thm} now follows from Proposition~\ref{prop:Rank2-formula} together with Lemma~\ref{lem:compute-degree}.
\end{proof}
\par
% steht schon oben
% \begin{rem}
% In the above proof, we showed that the monodromy transformation $\rho(\gamma)$ 
% about a cusp is always quasi-unipotent. In the context of VHS, 
% this is a theorem originally due to Borel (see e.g. \cite[Lemma 4.5]{schmid73}).
% \end{rem}

%**********************************
\section{Applications}\label{sec:examples}
%**********************************
In this section, we describe how to algorithmically obtain the monodromy representation in the case of origamis 
in terms of the action of generators of the affine group. Then we exhibit two principles to split up this representation
into subrepresentations. As an application, we present two examples where a splitting of the monodromy representation of a Teichmüller 
curve into rank $2$-subrepresentations is found. We then use the technique from the previous 
section to determine the Lyapunov spectrum. 
\par
\subsection{Algorithmic approach}
Given an origami $p: \Ori \to E$, we outline an algorithm for obtaining the monodromy representation of $\Aff(\Ori)$ in terms of its generators.  It has been realized mainly by Myriam Finster, building on work of Gabriela Weitze-Schmith\"usen, Karsten Kremer and others.
\par
To fix notations, let $E^*$ be $E$ minus the ramification point $e$ of $p$, and let $\Ori^* = \Ori\setminus p^{-1}(e)$. Then $p:\Ori^*\to E^*$ is a topological covering. We fix an isomorphism $\pi_1(E^*)\isom F_2$ by choosing the basis $x$,$y$ of $\pi_1(E^*)$ represented by a horizontal and a vertical path in $E^*$. The preimage of $x\cup y$ under $p$ is a $4$-valent graph $\Graph(\Ori^*)$ homotopy-equivalent to $\Ori^*$. Moreover, $\pi_1(\Ori^*)$ injects into $\pi_1(E^*)$; let $p_*$ be this injection and let its image be denoted by $H = H(\Ori^*)$.
\par
We make use of a proposition, which is already implicit in \cite{gabialgo}. Let $c: F_2 \to \Aut^+(F_2)$ denote the canonical inclusion of the inner automorphisms of $F_2$ into its orientation-preserving automorphisms and let $\beta:\Aut^+(F_2) \to \SL_2(\ZZ) \isom \Out^+(F_2)$ denote the canonical projection.
\begin{prop} \label{prop:affine-group-and-AutF2}
Let $p:\Ori\to E$ be an origami, and let $H= H(\Ori^*)$. There is a commutative diagram with exact rows
\begin{diagram}[height=4ex]
%  1 &\rTo &N(H) &\rTo^c   &\Stab^+(H) &\rTo^\beta &\SL(\Ori^*) &\rTo &1\\
%    &     &\dTo_\psi   &       &\dTo_\phi            &     &||    &\\
 1 &\rTo &N(H)/H &\rTo^c   &\Stab^+(H)/c(H) &\rTo^\beta &\SL(\Ori^*) &\rTo &1\\
   &     &\uTo_\isom   &       &\uTo_\isom^\psi           &     &||    &\\
 1 &\rTo &\Trans(\Ori^*) &\rTo &\Aff(\Ori^*)    &\rTo^\der &\SL(\Ori^*) &\rTo &1
\end{diagram}
where $\Stab^+(H)$ is the subset of $f\in \Aut^+(F_2)$ such that $f(H) = H$, and $N(H)$ is the normalizer of $H$ in $F_2$.
\par
Moreover, the injection $p_*$ is equivariant for the actions by outer automorphisms of $f\in \Aff(\Ori^*)$ on $\pi_1(\Ori^*)$ and of $\psi(f)\cdot c(H)\in \Stab^+(H)/c(H)$ on $H$.
\end{prop}
\par
Note that in general $\Aff(\Ori) \supsetneq \Aff(\Ori^*)$ if $\Ori$ is not a primitive origami.
 Also $\Trans(\Ori^*) = \Trans(\Ori)$ only holds if $g(\Ori) \geq 2$.
\par
\begin{proof}
 Let $u:\tilde X\to \Ori^*$ denote a fixed universal covering, and endow it with the translation structure obtained by pullback. Then $p\circ u:\tilde X \to E^*$ is a universal covering of $E^*$. Let $\Gal(\tilde X/E^*)$ denote the deck transformations of $p\circ u$. By \cite{gabialgo}, there is a commutative diagram with exact rows
\begin{gather}\label{diag:Gabi-F2-AutF2-Veech}
\begin{diagram}[height=4ex]
 1 &\rTo &F_2 &\rTo^c   &\Aut^+(F_2) &\rTo^\beta &\SL_2(\ZZ) &\rTo &1\\
   &     &\uTo_\isom   &       &\uTo_\isom            &     &||    &\\
 1 &\rTo &\Gal(\tilde X/ E^*) &\rTo &\Aff(\tilde X)    &\rTo^\der &\SL_2(\ZZ) &\rTo &1
\end{diagram}
\end{gather}
where the isomorphism $\Aff(\tilde X) \to \Aut^+(F_2)$ stems from the fact that each affine $f:\tilde X\to \tilde X$ descends to $E^*$ and induces an orientation preserving automorphism of $F_2$. Define $\Aff_u(\tilde X)$ to be the subgroup of affine automorphisms descending to $\Ori$ via $u$, and let $\Trans_u(\tilde X) = \Aff_u(\tilde X)\cap \Gal(\tilde X/E^*)$. We claim that we have a commutative diagram with exact rows
\begin{gather*}
\begin{diagram}[height=4ex]
 1 &\rTo &N(H) &\rTo^c   &\Stab^+(H) &\rTo^\beta &\SL(\Ori^*) &\rTo &1\\
   &     &\uTo_\isom   &       &\uTo_\isom            &     &||    &\\
 1 &\rTo &\Trans_u(\tilde X) &\rTo &\Aff_u(\tilde X)    &\rTo^\der &\SL(\Ori^*) &\rTo &1
\end{diagram}
\end{gather*}
The bottom row is exact by the definition of $\Trans_u(\tilde X)$ and the fact that the canonical projection $\Aff_u(\tilde X) \to \Aff(\Ori^*)$ is surjective. Again by \cite{gabialgo}, the image of $\Aff_u(\tilde X)$ in $\Aut^+(F_2)$ is precisely $\Stab^+(H)$ and the image of $\Stab^+(H)$ under $\beta$ is $\SL(\Ori^*)$. Finally, $c(F_2)\cap \Stab^+(H) = c(N(H))$.
\par
The first claim of the proposition now follows from the fact that the kernel of the canonical projection $\Aff_u(\tilde X)\to \Aff(\Ori^*)$ is precisely $\Gal(\tilde X/\Ori^*) \isom H$.
\par
As to the second claim, the description of the isomorphism $\Aff(\tilde X) \to \Aut^+(F_2)$ implies that $\psi(f)$ is the class of $\ol{f}_*$, where $\ol{f}$ is the map induced by $f$ on $E^*$. Thus for every path $\gamma \in \pi_1(\Ori)$, $\psi(f)(p_*\gamma) = \ol{f}_*p_*\gamma = p_*f_*\gamma$. Since $p_*$ is an isomorphism onto its image, it maps the conjugacy class of $f_*\gamma$ in $\pi_1(\Ori)$ to the conjugacy class of $\psi(f)(p_*\gamma)$, which proves the claim.
\end{proof}
\par
The input of our algorithm is an origami $p:\Ori\to E$ of degree $\deg p = d$ and genus $g$, given as graph $\Graph(\Ori^*)$.
\par
\begin{itemize}
 \item[\textit{Step 1:}] Construct a basis of $\pi_1(\Ori^*)$. Choose a maximal spanning tree $T$ in $\Graph(\Ori^*)$. The edges $t_1,\dots,t_{d+1}$ not in $T$ represent a basis of $\pi_1(\Ori^*)$. Mapping this basis to $H\subgp F_2$, we obtain a free system of generators $u_1,\dots,u_{d+1}$ for $H$.
 \item[\textit{Step 2:}] Compute a system of generators $\gamma_1,\dots,\gamma_r$ of $\Stab^+(H)$ (see \cite{FinsterDA}).
 \item[\textit{Step 3:}] Lift the action of $\gamma_i$ on the generators of $H$ to an action on $t_1,\dots,t_{d+1}$. Let $w_{ij} = \gamma_i(u_j)$; this is a word  in $x$, $y$ which can be decomposed as a word in the generators of $\pi_1(\Ori^*)$ by writing down all non-tree edges crossed on the path in $\Graph(\Ori^*)$ determined by $w_{ij}$.
 \item[\textit{Step 4:}] Find an extended symplectic basis $a_1,b_1,\dots,a_g,b_g,c_1,\dots,c_{m-1}$ of $\pi_1(\Ori^*)$ by surface normalization as in \cite{Stillwell80}. Here, the $c_i$ are loops about all but one puncture in $p^{-1}(e)$.
 \item[\textit{Step 5:}] For each generator $\gamma_i$ of $\Stab^+(H)$, project its action on the generators of $\pi_1(\Ori^*)$ to $\GL(H_1(\Ori^*,\ZZ))$. Then make a base change to the extended symplectic basis found in Step 4. Discard the basis elements representing loops around the punctures to obtain the action of $\Aff(\Ori^*)$ on $H_1(\Ori,\ZZ)$.
\end{itemize}
\par
Proposition \ref{prop:affine-group-and-AutF2} implies the correctness of the above algorithm.
\par
The action of $\Aff(\Ori^*)$ on $H^1(\Ori,\ZZ)$ is obtained by using the duality of $H_1$ and $H^1$. Note that if $\gamma$ acts by $A\in \Sp(2g,\ZZ)$ \wrt a symplectic basis of $H_1(\Ori,\ZZ)$, then $(A^{-1})^T$ is the matrix of the left action of $\gamma$ on $H^1(\Ori,\ZZ)$ \wrt the dual basis. While there is no substantial difference between the action of $\Aff(\Ori)$ on homology and on cohomology, we prefer to work with cohomology, since it exhibits a better functorial behavior.

%**********************************
\subsection{Splitting principles}
%**********************************
We describe two principles for finding subrepresentations of a monodromy representation. 
\par
Given two Veech surfaces $(X,\omega)$, $(Y,\nu)$, we call a non-constant
 holomorphic map $f:X\to Y$ a \textit{Veech covering} if $f^*\nu = \omega$ and 
if the Veech group of $Y$ minus the ramification points of $f$ is a lattice. Note that this happens 
if and only if all branch points are periodic points, \ie have finite $\Aff(Y,\nu)$-orbits. 
\par
A Veech covering $p:(X,\omega)\ra (Y,\nu)$ between Veech surfaces induces a subrepresentation as follows. By \cite[Theorem 4.8]{gutkinjudge} the elements of $\Aff(X,\omega)$ that descend via $p$ to $Y$ form a finite-index subgroup $\Aff(X,\omega)_p$ of $\Aff(X,\omega)$. Let 
\[\varphi_p: \Aff(X,\omega)_p\ra \Aff(Y,\nu)\]
 be the group homomorphism that maps $f\in \Aff(X,\omega)_p$ to $\ol{f} \in \Aff(Y,\nu)$ such that $p\circ f = \ol{f}\circ p$. The image of $\varphi_p$ is the finite-index subgroup $\Aff(Y,\nu)^p$ of $\Aff(Y,\nu)$ of affine diffeomorphisms, that lift to $(X,\omega)$.
\par
\begin{prop}\label{prop:Splitting by covering map}
Let $p:(X,\omega)\ra (Y,\nu)$ be a Veech covering between Veech surfaces and let 
$\rho:\Aff(X,\omega)\ra \Sp(H^1(X,\ZZ))$ be the monodromy representation of $(X,\omega)$. Then the image $U$ of $H^1(Y,\ZZ)$ under \[p^*:H^1(Y,\ZZ)\ra H^1(X,\ZZ)\]
is an $\Aff(X,\omega)_p$-invariant symplectic subspace of $H^1(X,\ZZ)$ polarized by $\deg(p)\cdot Q_X$.
\par
The map $p^*$ is equivariant for the action of $\Aff(X,\omega)_p$ on $U$ and $\Aff(Y,\nu)^p$ on $H^1(Y,\ZZ)$.
\end{prop}
\begin{proof}
 Let $f\in \Aff(X,\omega)_p$ and $\ol{f}\in \Aff(Y,\nu)$ such that $p\circ f = \ol{f} \circ p$. Then for every $c\in H^1(Y,\ZZ)$
\[(f^{-1})^*(p^*(c)) =  (p\circ f^{-1})^*(c) = (\ol{f}^{-1}\circ p)^*(c) = p^*((\ol{f}^{-1})^*(c))\ ,\]
proving $(f^{-1})^*(\Image(p^*))\subset \Image(p^*)$. The computation also shows that $p^*$ is equivariant. Finally, $p^*$ is a symplectic map and $Q_X(p^*c_1,p^*c_2) = \deg p\cdot Q_Y(c_1,c_2)$.
\end{proof}
\par
We note that the uniformizing subrepresentation of an origami is induced by the Veech covering $p:\Ori \to E$.
\par
Secondly, the group $\Aut(X,\omega)$ of affine biholomorphisms acts on $H^1(X,\RR)$ and $H^1(X,\CC)$ and we can use representation theory of finite groups to decompose these vector spaces into a direct sum $\K[\Aut(X,w)]$-modules (with $\K = \RR$ or $\CC$). This technique has been successfully applied in \cite{MYZHomology}.
\par
\begin{prop}\label{prop:Splitting by Galois group}
Let $(X,\omega)$ be a Veech surface and let $G \subgp \Aut(X,\omega)$. The action of $\Aff(X,\omega)$ on $H^1(X,\K)$, restricted to the normalizer $N(G)$ of $G$ in $\Aff(X,\omega)$, permutes the isotypic components of the decomposition of $H^1(X,\K)$ into $G$-modules and there is a finite index subgroup $\Gamma\subgp \Aff(X,\omega)$ such that every isotypic component is $\Gamma$-invariant. 
\end{prop}
\begin{proof}
As $\Aut(X,\omega)$ is normal in $\Aff(X,\omega)$, the normalizer $N(G)$ of $G$ in $\Aff(X,\omega)$ has finite index in $\Aff(X,\omega)$. For all $g\in G$, and $f\in N(G)$, there exists $\tilde g\in G$, such that $gf = f\tilde g$. Therefore for all irreducible $\K[G]$-submodules $V$ of $H^1(X,\K)$, we have 
\[(g^*)^{-1}\circ (f^*)^{-1}(V) = ((gf)^*)^{-1}(V) = ((f\tilde g)^*)^{-1}(V) = (f^*)^{-1}(V),\]
which shows that $(f^{-1})^*(V)$ is another irreducible $\K[G]$-module inside $H^1(X,\K)$. Hence every $f\in N(G)$ induces a permutation of the isotypic components of the representation of $G$. Thus there is a finite index subgroup $\Gamma\subgp N(G)$ that leaves every isotypic component invariant.
% For the remaining assertions, recall that a subspace of $H^1(X,\C)$ of the form $V\oplus \ol{V}$ with $V\subseteq \Omega^1_X(X)$ is symplectic by virtue of Remark \ref{rem:Hodge norm on H10}.
\end{proof}
\par
% Needed?
% Note that we can also find a finite-index subgroup $\Gamma\subgp \Aff(X,\omega)$ such that each element $\gamma\in \Gamma$ acts on each isotypic component $V$ of the decomposition of $H^1(X,\K)$ as a $\K[G]$-linear automorphism. Indeed, it suffices to take $\Gamma = \bigcap_{g\in G}C(g)$, where $C(g)$ is the centralizer of $g$ in $\Aff(X,\omega)$.
\begin{rem}
 In both cases, the subrepresentations carry a VHS. This follows directly from Deligne's semisimplicity theorem.
\end{rem}
\par
There can be invariant subspaces not directly related to these two constructions due to hidden symmetries of the Jacobian (\eg endomorphisms of Hecke type as discussed in \cite{EllenbergEndoJac}).

%**********************************
\subsection{Examples}
%**********************************
The examples discussed in the following are both origamis and stem from intermediate covers of the characteristic origami $\CharL22$
discussed in \cite{Her06}. We remain rather brief here; a complete discussion including all matrix computations is found in \cite{KappesThesis}.
\par
We note that in our examples the individual Lyapunov exponents can also be obtained from the formula for their sum, combined with knowledge on intermediate coverings.
\par
%this is mentioned above!
% On an appropriate finite cover, we have a VHS and consequently a period map. 
% This period map descends to the orbifold curve, and we can compute its degree
% there instead.
In the following, denote $T = \textmatrix{1}{1}{0}{1}$ and $S =\textmatrix{0}{-1}{1}{0}$.

\subsubsection*{First example}
Let $\L22$ be the origami given by
\[r = (1\ 2)(3),\qquad u = (1\ 3)(2),\]
where the permutation $r$ ($u$) is the monodromy the horizontal (vertical) generator of $\pi_1(E^*)$.
It is the smallest origami of genus 2. Its affine group $\SL(\L22)$ is
isomorphic to the index $3$ subgroup $\Gamma_{\Theta}$ of $\SL_2(\ZZ)$ generated by $S$, $T^2$. 
It follows that $\SL(\L22)$
is isomorphic to the orientation preserving subgroup of the $\Delta(2,\infty,\infty)$-triangle group
of the hyperbolic triangle with vertices $(i,\infty,1)$. In particular, the stabilizer
of $1$ is generated by $T^{2}S$.
\begin{figure}[t]
\begin{center}
\raisebox{20mm}{\mbox{
% Origami of size 3
\begin{xy}
<0.8cm,0cm>:
(4,-1)*{\OriSquare{1}{}{1}{2}{1}};
(5,-1)*{\OriSquare{2}{1}{}{}{3}};
(5,0)*{\OriSquare{3}{3}{2}{3}{}};
\end{xy}
}}\hspace*{15mm}
{\mbox{\includegraphics[scale=0.5]{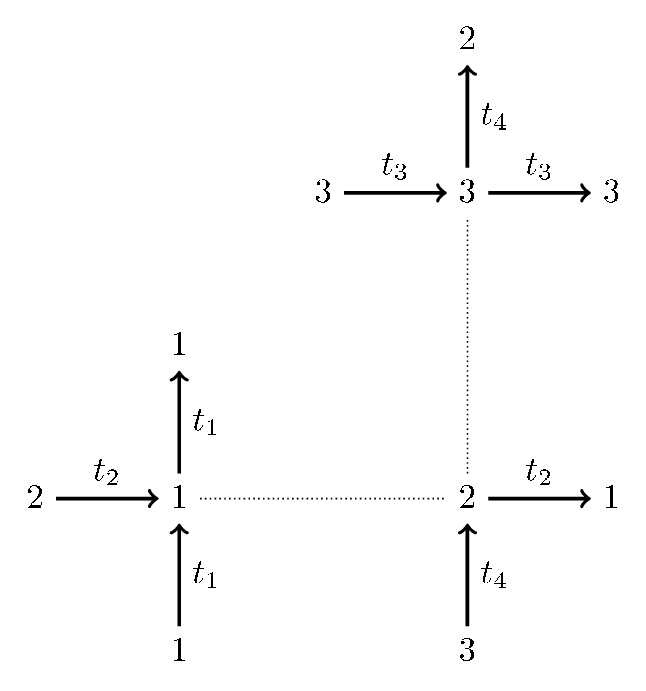}}}
\end{center}
\begin{center}
 \caption{\textit{The origami $\L22$ and a maximal spanning tree of $\Graph(\L22^*)$} \label{fig:origami L22}}  
\end{center}
\end{figure}
\par
To analyze the action of $\Aff(\L22)$ on $H^1(\L22,\ZZ)$, let $t_1,\dots,t_4$ be the basis of $\pi_1(\L22^*)$ associated with the non-tree edges of a maximal spanning tree as in Figure \ref{fig:origami L22}. Then
\[a_1 = \ol{t}_1, \quad b_1 = -\ol{t}_2,\quad a_2 = \ol{t}_3,\quad b_2 = \ol{t}_4 - \ol{t}_1\]
is a symplectic basis of $H_1(\L22,\ZZ)$. Let further
\[h = \ol{t}_3 + \ol{t}_2 = a_2 - b_1 \quad \text{and}\quad v = \ol{t}_1 + \ol{t}_4 = 2a_1 + b_2\]
be the sum of all horizontal, respectively vertical cycles. The action of $\Aff(\L22)$ on $H^1(\L22,\ZZ)$ splits over $\QQ$ into two 2-dimensional representations. 
The uniformizing representation is spanned by the image of $h$ and $v$ 
in $H^1(\L22,\ZZ)$ under $a\mapsto i(\cdot, a)$, where $i(\cdot,\cdot)$ denotes the symplectic intersection form on homology. 
The representation $\rho_{\L22,2}:\Aff(\L22) \to \SL_2(\ZZ)$, complementary to the uniformizing representation, is given by
\[T^2\mapsto T,\quad S\mapsto S^{-1}\]
with respect to the basis
\[a_1^* - 2b_2^*, \quad b_1^* + a_2^*.\]
\begin{prop}
 The non-negative Lyapunov exponent associated with $\rho_{\L22,2}$ is $1/3$.
\end{prop}
\begin{proof}
 Let $p: \HH\to \HH$ denote the period map of the VHS associated with $\rho_{\L22,2}$. Since $T^2 \mapsto T$, while $T^2S\mapsto TS^{-1}$, an element of order $3$, the preimage of the cusp $i\infty$ of $\SL_2(\ZZ)$ under $p$ is only the cusp $i\infty$. By Lemma \ref{lem:compute-degree}, $\deg \ol{p} = 1$. By Proposition \ref{prop:Rank2-formula} and Lemma \ref{lem:going up},
\[\lambda = \frac{\vol(\HH/\SL_2(\ZZ))}{\vol(\HH/\Gamma_{\Theta})} = 1/3.\]
\end{proof}
\par
Note that this matches Bainbridge's result on Lyapunov exponents of invariant measures on $\Stratum_2$ \cite{Bainbridge07}.

% 
% A $\Delta(m,n,p)$ triangle group has presentation
% \[\pres{a,b,c}{a^2,b^2,c^2, (ab)^m, (bc)^n, (ca)^p}\]
% Its orientation preserving subgroup is the index 2-subgroup generated
% by all words of even length
% \[\Delta^+(m,n,p) = \pres{x,y}{x^m,y^n,(xy)^p}\]
% with $x = ab$, $y=bc$.
\begin{figure}[ht]
\begin{center}
\mbox{
\begin{xy}
<1cm,0cm>:
(1,-2)*{\OriSquare{1}{4}{6}{7}{}};
(1,-3)*{\OriSquare{2}{}{}{9}{3}};
(2,-3)*{\OriSquare{3}{}{2}{}{7}};
(3,-2)*{\OriSquare{4}{7}{9}{1}{}};
(3,-3)*{\OriSquare{5}{}{}{}{9}};
(4,-3)*{\OriSquare{6}{}{8}{}{1}};
(5,-2)*{\OriSquare{7}{1}{3}{4}{}};
(5,-3)*{\OriSquare{8}{}{}{}{6}};
(6,-3)*{\OriSquare{9}{2}{5}{}{4}};
(0.5,-1.45)*{\times};
(0.5,-2.45)*{\times};
(2.5,-3.45)*{\times};
(3.5,-3.45)*{\times};
(5.5,-2.45)*{\times};
(5.5,-1.45)*{\times};
(6.5,-2.45)*{\times};
(1.5,-1.45)*{\circledast};
(2.5,-1.45)*{\circledast};
(3.5,-2.45)*{\circledast};
(4.5,-2.45)*{\circledast};
(4.5,-3.45)*{\circledast};
(5.5,-3.45)*{\circledast};
(0.5,-3.47)*{\bullet};
(1.5,-2.47)*{\bullet};
(1.5,-3.47)*{\bullet};
(2.5,-2.47)*{\bullet};
(3.5,-1.47)*{\bullet};
(4.5,-1.47)*{\bullet};
(6.5,-3.47)*{\bullet};
\end{xy}}
\end{center}
\begin{center}
 \caption{\textit{The origami $\Qmod9$} \label{fig:origami M}}  
\end{center}
\end{figure}
\par
\begin{figure}[ht]
\begin{center}
\includegraphics[scale=1]{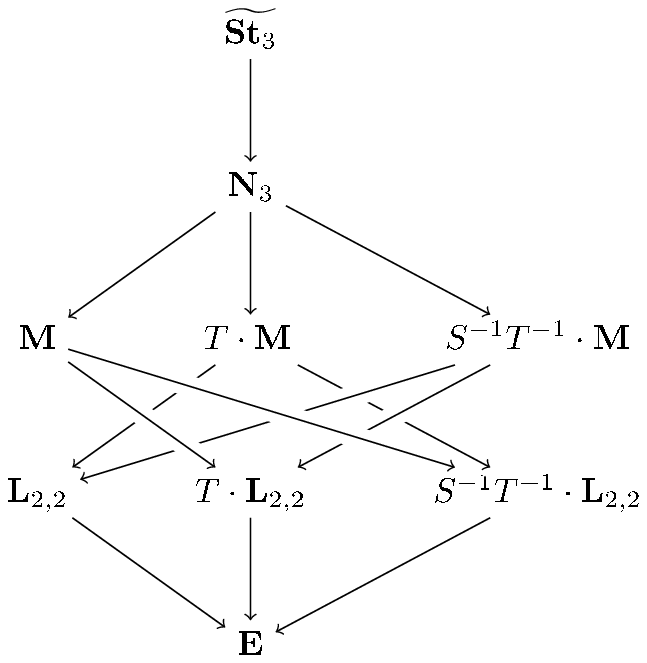}
\end{center}
\begin{center}
 \caption{\textit{Intermediate covers of $\CharL22$. $\SL_2(\ZZ)$-orbits of $\L22$, $\Qmod9$ and a third origami $\Mod27_3$ with $27$ squares and Veech group $\SL_2(\ZZ)$. Arrows indicate Veech covering maps} \label{fig:poset}}  
\end{center}
\end{figure}
\par
\subsubsection*{Second example}
The second example is the origami $\Qmod9$ (see Figure \ref{fig:origami M}) given by 
\[r = (1,4,7)(2,3,5,6,8,9)\quad \text{and}\quad u = (1,6,8,7,3,2)(4,9,5),\]
and belongs to $\Stratum_4(2,2,2)^{\Odd}$. Here, ``odd'' refers to the connected component of surfaces of odd spin structure with respect to the classification of connected components of strata in \cite{KontsevichZorich03}.
The affine group of $\Qmod9$ is also equal to
$\Gamma_{\Theta}$. The monodromy representation 
$\rho_{\Qmod9} : \Aff(\Qmod9) \to H^1(\Qmod9,\ZZ)$ restricted 
to $\Gamma(2) \subgp \SL(\Qmod9)$ splits over $\QQ$ into four symplectic 
subrepresentations, each of rank $2$. Apart from the uniformizing representation $\rho_1$, 
there are two representations
$\rho_{2}$, $\rho_{3}$ that are induced from coverings to genus 2-origamis. 
Figure \ref{fig:poset} shows the poset of intermediate covers of $\CharL22$.
\par
Let us denote $\rho_{\Qmod9,4}$ the representation complementary to 
$\rho_{\Qmod9,1}\oplus \rho_{\Qmod9,2}\oplus \rho_{\Qmod9,3}$. It already splits off over $\SL(\Qmod9)$, 
and is given by
\[\rho_{\Qmod9,4}(T^2) = T^{-1}S = \begin{pmatrix}
                                    -1 & -1  \\ 1 & 0
                                   \end{pmatrix}
,\qquad \rho_{\Qmod9,4}(S) = S^{-1}.\]
\begin{prop}\label{prop:exa2-Lyapspectrum}
The non-negative Lyapunov spectrum of $\Qmod9$ is
\[1,\tfrac{1}{3}, \tfrac{1}{3}, \tfrac{1}{3}.\]
\end{prop}
\begin{proof}
 By Proposition \ref{prop:Splitting by covering map}, each of the coverings to the genus 2-origamis induces a rank 4-subspace invariant under some finite index subgroup $\Gamma$ (we can take $\Gamma = \Gamma(2)$) of $\Aff(\Qmod9)$. The pullback of the uniformizing representation is the uniformizing representation upstairs. Furthermore a computation shows that the pullbacks of the non-uniformizing representations are distinct, whence two rank 2-representations $\rho_{\Qmod9,2}$ and $\rho_{\Qmod9, 3}$. Both being pulled back from origamis in $\Stratum_2(2)$, they have Lyapunov exponent $1/3$. As in the first example, the third Lyapunov exponent can be shown to be also $1/3$.
\end{proof}
\par 
\begin{rem} The representation, although not induced via a Veech covering from genus $2$, is not very far from the representation $\rho_{\L22,2}$. More precisely, $\rho_{\L22,2}$ is taken to
$\rho_{\Qmod9,4}$ by the orientation-reversing outer automorphism 
\[\alpha: (T^2,S)\mapsto(T^{-2}S^{-1},S)\]
of $\Gamma_{\Theta}$, \ie $\rho_{\L22,2}\circ \alpha = \rho_{\Qmod9,4}$.
\end{rem}

\section{Modular embeddings of rank 2}\label{sec:period data}
%%%%%%%%%%%%%%%%%%%%%%%%%%%%%%%%%%%%%%%%
% The motivation of this section comes from a theorem by Deligne, which can be stated in a special case as follows.
% \begin{prop}[{\cite{delfinitude}}]
% For a fixed algebraic curve $C$ and $N\in \NN$, there are only finitely many isomorphism classes of irreducible representations 
% of $\pi_1(C,c)$ of rank $N$ occuring as a direct factor in monodromy representations of families of curves $\phi:\XFamily\to C$.
% \end{prop}
% [Coefficients? Deligne says, $\QQ$, but also $\ZZ$ possible... what about $\CC$ or $\RR$?
% klar: Jede irred. Q-Darstellung zerfällt in endlich viele C-Darstellungen => nur endlich viele 
% versch. C-Darstellungen

As we have seen in Section \ref{sec:background}, a variation of Hodge structures can equally well be 
described as a group homomorphism $\rho$ plus the $\rho$-equivariant period map.
In this section, we study these objects from an abstract point of view and we
exhibit rigidity properties.
\par
In the following, let $G= \PSL_2(\RR)$.
\begin{defn}
A \textit{modular embedding} (of rank 2 and weight 1) is a pair $(p,\rho)$, where
\begin{enumerate}[(i)]
 \item $\rho:\Gamma \to G$ is a
group homomorphism from a 
lattice $\Gamma\subgp G$.
% group $\Gamma$ acting cofinitely and holomorphically on $\HH$
%  with finite kernel\todo{needed? Yes: only affine group acts on cohomology, not Veech group.
% finite kernel needed for commensurability question}
 \item $p:\HH\to \HH$ is a $\rho$-equivariant holomorphic
map.
\end{enumerate}
\end{defn}
\par
% \par
% \begin{defn}
% A \textit{period datum} (of rank 2 and weight 1) is a pair $(p,\rho)$, where
% \begin{enumerate}[(i)]
%  \item $\rho:\Gamma \to \SL_2(\RR)$ is a
% group homomorphism from a group $\Gamma =\dom(\rho)$ acting cofinitely by isometries on $\HH$ with finite kernel
%  \item $p:\HH\to \HH$ is a $\rho$-equivariant holomorphic
% map.
% \end{enumerate}
\begin{defn}
For a modular embedding $(p,\rho)$, denote by $\dom(\rho)$ the domain and by $\range(\rho)$ the image 
of $\rho$. We call a modular embedding \textit{discrete} if $\range(\rho)$ is a discrete subgroup of 
$G$ 
% $\SL_2(\RR)$
, and \textit{cofinite} if 
$\range(\rho)$ acts discretely and
cofinitely on $\HH$. 
If $\range(\rho)$ acts cofinitely, then $p$ descends to a holomorphic map between the quotients, which we denote by $\ol{p}$.
\end{defn}
\par
Note also that we allow $\Gamma$ to contain torsion elements in order to handle
 the orbifold case.
%  In particular, a modular embedding defines a VHS only on an appropriate
% finite cover $\HH/\Gamma'$.
\par
\begin{rem}
As in the proof of Theorem~\ref{main thm}, one shows that if a modular embedding is
 discrete, then it is either constant (\ie $p$ is constant) or cofinite.
\end{rem}
\par
Examples of modular embeddings come from Teichm\"uller curves.
Apart from the examples given above, there are prominent ones in 
$\Moduli_2$ discovered in \cite{mcmullenbild}. 
Here, $\SL(X,\omega)$ injects into $\SL_2(\mathfrak{o}_D)$ 
for some order $\mathfrak{o}_D$ in a totally real 
quadratic number field $\QQ(\sqrt{D})$. The VHS splits into two sub-VHS of rank 2, and the
period map of the non-uniformizing sub-VHS, together with the representation of $\SL(X,\omega)\isom \Aff(X,\omega)$ given by 
Galois conjugation give rise to a modular embedding. Other examples related to these
 are the twisted Teichm\"uller curves studied in \cite{weiss}. 
% maybe in the future: add: if $p$ is iso then this is the case of HZ-cycles}
\par
%******************
\subsection{Rigidity}
%******************
In this section, we gather results on how much the two data $\rho$ and $p$ of a modular embedding determine each other.
\par
If $p$ is non-constant, it is easy to see that the representation of a modular embedding $(p,\rho)$ is uniquely determined by $p$.
% as a representation to $\PSL_2(\RR)$.
Conversely, the period map is also uniquely determined by the representation. This has already been remarked by McMullen \cite[Section 10]{mcmullenbild}, and can in fact be generalized to ball quotients \cite[Theorem 5.4]{KapMoe12}. We recall the arguments for the convenience of the reader.
\par
\begin{prop}\label{prop:uniqueness of period map}
%  Let $\Gamma$ be a cofinite Fuchsian group and let $\rho:\Gamma\to \SL_2(\RR)$ be a group homomorphism such that $\rho(\Gamma) \nsubseteq \{\pm I\}$. 
Given a non-trivial group homomorphism $\rho:\Gamma \to G$ from a cofinite Fuchsian group $\Gamma$, there exists
 at most one map $p:\HH\to\HH$ such that $(p,\rho)$ is a modular embedding.
\end{prop}
\par
\begin{proof}
We work in the unit disk model and use arguments displayed in \cite[Section 2]{Shiga}. Since $\Gamma$ is a lattice, it is of divergence type. Therefore the set of points $E$ in $\partial \DD = S^1$, which can be approximated by a sequence $(\gamma_k(x_0))_k \subset \Gamma$ (for some $x_0\in \DD$) that stays in an angular sector, is of full Lebesgue-measure in $\RR\cup \{\infty\}$.
For a holomorphic map $p: \DD\to \DD$ define $p^*(\zeta)$ of a point of approximation $\zeta\in E$ by $\lim_k p(\gamma_k(x_0))$ for a sequence $\gamma_k(x_0) \to \zeta$. This is well-defined for almost all $\zeta$ and $p^*(\zeta)\in \partial \DD$ for almost all $\zeta$ by \cite[Lemma 2.2]{Shiga}.
\par
Now suppose we are given two $\rho$-equivariant maps $p_i$, $i=1,2$. Pick a point $x_0 \in \DD$. If $p_1$ is constant then $\rho(\Gamma)$ lies in the stabilizer of $p_1(x_0)$. By equivariance, $p_2(y)$ is stabilized by $\rho(\Gamma)$ for any $y\in \DD$. Since $\rho$ is non-trivial, $p_1=p_2$. Thus we are left with the case that $p_1$, $p_2$ are non-constant. Then for all $k$
\[d_{\DD}(p_1(x_0),p_2(x_0)) = d_{\DD}( p_1(\gamma_kx_0), p_2(\gamma_kx_0))\]
and since $p_i(\gamma_kx_0) \to \partial \DD$, this means that $p_1^*(\zeta) = p_2^*(\zeta)$ for $\zeta$ in a set of full measure of $\partial \DD$. Thus $(p_1-p_2)^* = p_1^*-p_2^* \equiv 0$ and therefore $p_1 = p_2$.
\end{proof}
\par
If $(p,\rho)$ is cofinite, then it determines a map $\ol{p}:\HH/\dom(\rho) \to \HH/\range(\rho)$. Conversely, a map $\ol{p}$ between the quotients gives rise to a modular embedding as the following lemma shows. Thus there are in some sense many modular embeddings.
\par
\begin{lem}\label{lem:given pbar-rho exists}
 Let $\bar{p} : \HH/\Gamma \to \HH/\Delta$ be a non-constant holomorphic map
between finite-area Riemann surfaces. Denote $u_{\Delta}:\HH\to \HH/\Delta$ the
canonical projection, and let $z\in \bar{p}(\HH/\Gamma)$. 
If $\Delta\subset
% \PSL_2(\RR)$ 
G$
acts freely on $\HH$, then $\bar{p}$ lifts to a holomorphic map
$p:\HH\to\HH$, unique up to the choice of a point $\tilde z\in u^{-1}(z)$, and
there is a unique group homomorphism
\[\rho:\Gamma \to \Delta\]
such that $p$ is $\rho$-equivariant.
\end{lem}
\par
We suspect that this statement is well-known, but we are not aware of a source. We supply a proof for the convenience of the reader.
\par
\begin{proof}
 The first claim follows from $u_{\Delta}$ being a covering map.
\begin{draftcomment}
 We can lift $\bar{p}\circ u_\Gamma:\HH\to \HH/\Delta$ to $p:\HH\to\HH$,
 since $\HH$ is simply connected.
\end{draftcomment}
 As to the second claim, note that $\Delta$ acts freely and transitively on
$u_{\Delta}^{-1}(z) = \Delta \cdot \tilde z$. Let $y\in\bar{p}^{-1}(z)$, and let
$\tilde y\in u_{\Gamma}^{-1}(y)$ with $p(\tilde y) = \tilde z$. (This fixes
$p$.) For any $\gamma\in \Gamma$, there is by assumption a unique
$\delta_{\gamma,\tilde y}\in \Delta$ such that \[p(\gamma \tilde y) = \delta_{\gamma,\tilde y}
p(\tilde y).\]
Thus we can define a map $\rho:\Gamma \to \Delta, \rho(\gamma) =
\delta_{\gamma,\tilde y}$. To check that $\rho$ is a group homomorphism, we first show
that if $c:[0,1]\to \HH$ is a path starting at $\tilde y$, then $p(\gamma c(t))
= \delta_{\gamma,\tilde y} p(c(t))$ for all $t$. We know
that for each $t$ there is $\delta_{\gamma,c(t)}$ such that
$p(\gamma c(t)) = \delta_{\gamma,c(t)} p(c(t))$. On the other hand, we claim that the assignment $t\to\delta_{\gamma,c(t)}$ is 
locally constant, and hence constant since $[0,1]$ is connected: Each $\tilde x\in \HH$ has a neighborhood $U$ 
such that for all $\tilde w\in U$, $p(\gamma \tilde w) = \delta' p(\tilde w)$ only holds for $\delta' = \delta_{\gamma,\tilde x}$.
Indeed, it suffices to take $U = p^{-1}(V)$, where $V$ is a neighborhood of $p(\tilde x)$ such that
$\delta V \cap V = \emptyset$ for all $\delta\in \Delta$, $\delta \neq \id$.
% Let $s,t\in [0,1]$. Then
% \[|\delta_t(p(c(t)) - \delta_s(p(c(t))| \leq |p(\gamma c(t)) - p(\gamma c(s))|
% + |\delta_s(p\circ c(s)) - \delta_s(p\circ c(t))|.\]
% As $s\to t$, the expression on the right tends to $0$. 
% why this??	
% By discreteness of the
% orbit $\delta_s(p\circ c(t)) = \delta_t(p\circ c(t))$ for $|t-s|$ small. 
% Since
% the action is free, $\delta_s = \delta_t$ in a neighborhood of $t$. 
% Thus, the
% assignment $t\to \delta_t$ is locally constant, hence constant since $[0,1]$ is
% connected.
\par
This shows that $p(\gamma \gamma' \tilde y) = \delta_{\gamma,\tilde y} p(\gamma' \tilde y)$
for $\gamma,\gamma'\in \Gamma$: we take $c$ to be a path connecting $\tilde y$
and $\gamma'\tilde y$. Thus we have
\[p(\gamma\gamma' \tilde y) = \delta_{\gamma,\tilde y} p(\gamma'\tilde y) = \delta_{\gamma,\tilde y}
\delta_{\gamma',\tilde y} p(\tilde y) = \delta_{\gamma\gamma',\tilde y}p(\tilde y)\]
by uniqueness. The uniqueness of $\rho$ follows directly from the construction.
\end{proof}
\par
\begin{rem}
Given a modular embedding, we can consider the case when one of the two items is an isomorphism.
If $p = A\in G$ is a M\"obius transformation, then clearly $\rho$ is conjugation by $A$. Conversely, suppose 
$\rho$ is an isomorphism. If $(p,\rho)$ is cofinite, then after passing to a finite index subgroup, we can suppose that
$\Gamma$ is torsionfree and that $g(\HH/\Gamma)>1$. Then $\ol{p}:\HH/\Gamma \to \HH/\rho(\Gamma)$ must have degree $1$ by
the Riemann-Hurwitz formula, and hence $p$ is an isomorphism. If $(p,\rho)$ is not cofinite, it may well happen however that
$\rho$ is an isomorphism without $p$ being one. Examples are provided by Teichm\"uller curves in $g=2$ for non-square discriminants 
where $\rho$ is induced by Galois conjugation, but $p$ is not an isometry (see \cite[Theorem 4.2]{mcmullenbild}).
\end{rem}
\par
Using the previous lemma, we can now pick up the discussion from the introduction and show that every rational number in $[0,1]$ is the Lyapunov exponent of a family of elliptic curves.
\par
\begin{prop}\label{prop: every ratl is lyap}
 For any rational number $0\leq \lambda\leq 1$, there is a family of elliptic curves $\phi:\XFamily\to \HH/\Delta$ such that $\lambda$ is in the Lyapunov spectrum of its VHS.
\end{prop}
\begin{proof}
Let $\Gamma(2) = \ker(\SL_2(\ZZ)\to \SL_2(\ZZ/(2)))$ and let $P\Gamma(2)$ be its projection to $\PSL_2(\RR)$. We construct a holomorphic map 
\[\ol{p}: X\to \HH/P\Gamma(2) \isom \PP^1\setminus \{0,1,\infty\}\]
of degree $d$ from a Riemann surface $X$ by specifying a monodromy.
The map $p$ should be ramified over the cusps and over $r$ interior points $x_1,\dots,x_r$ 
in such a way that the associated covering is connected and $|\ol{p}^{-1}(x_i)| = t_i$. We can surely 
find such a monodromy 
\[\sigma :\pi_1(\PP^1\setminus \{0,1,\infty,x_1,\dots,x_r\}) \to S_d,\]
 since the fundamental group is free of rank $r+2$ (to guarantee connectedness, we can take $\ol{p}$ to be totally ramified over $\infty$).
\par
Next we choose a lattice $\Delta \subgp \PSL_2(\RR)$ such that $X \isom \HH/\Delta$. 
Since $P\Gamma(2)$ is torsionfree, we obtain a group homomorphism $\rho:\Delta\to P\Gamma(2)$ by Lemma \ref{lem:given pbar-rho exists}. We can lift this homomorphism to $\tilde\rho:\Delta \to \Gamma(2)^+$, where $\Gamma(2)^+$ is the group generated by $\textmatrix{1}{2}{0}{1}$ and $\textmatrix{1}{0}{2}{1}$, an index $2$ subgroup in $\Gamma(2)$: for $a\in \Delta$ we let $\tilde\rho(a)$ be the unique lift of $\rho(a)$ to $\SL_2(\RR)$ that is in $\Gamma(2)^+$.
\par
Let $p$ be the lift of $\ol{p}$. The pair $(p,\tilde \rho)$ is then a modular embedding, and the associated VHS is the VHS of a family $\phi:\XFamily\to \HH/\Delta$ of elliptic curves. In fact, $\phi$ is the pullback via $p$ of the universal family over $\HH$. By Proposition \ref{prop:Rank2-formula}, its sole non-negative Lyapunov exponent is given by
\[\lambda = \frac{\deg(\ol{p})\vol(\HH/\Gamma(2))}{\vol(\HH/\Delta)} = \frac{\deg(\ol{p})\chi(\HH/\Gamma(2))}{\chi(\HH/\Delta)}.\]
We have $\chi(\HH/\Gamma(2)) = -1$ and 
$\chi(\HH/\Delta)$ and $d = \deg(p)$ are related by the Riemann-Hurwitz formula
\[-\chi(\HH/\Delta) = 2g(\HH/\Delta) - 2 + s(\Delta) = d + \sum_{i=1}^r d-t_i = d(r+1) - \sum_{i=1}^r t_i,\]
where $s(\Delta)$ is the number of cusps of $\Delta$. Therefore,
\[\lambda = \biggl(r+1 - \frac{\sum_i t_i}{d}\biggr)^{-1}\]
where $t_i\in \{1,\dots,d\}$. For fixed $r,d$, the possible values of $\lambda^{-1}$ are
\[\set{r+1 - \frac{l}{d}}{l = r,r+1,\dots,rd}.\]
Letting $r$ and $d$ vary, we can thus obtain every rational number $\geq 1$. Hence every $\lambda \in \QQ\cap (0,1]$ can be realized as Lyapunov exponent. Finally, $\lambda = 0$ is the Lyapunov exponent of a constant family of elliptic curves. 
\end{proof}
\par
%***********************************************************
\subsection{Commensurability and Lyapunov exponents}
%***********************************************************
We define (weak) commensurability of two modular embeddings and show that the Lyapunov exponent of a modular embedding is a weak commensurability invariant. Further, we define the commensurator $\Comm(p,\rho)$ of a modular embedding in analogy to the usual commensurator. If $\rho$ has a non-trivial kernel, we show that $\dom(\rho)$ is of finite-index in $\Comm(p,\rho)$.
% and discuss its relationship to the commensurator of $\dom(\rho)$. We exhibit an analogy with Margulis' commensurability critertion which states that arithmetic lattices are characterized by having infinite index in their commensurator. Here the difference is that we find modular embeddings such that $\dom(\rho)$ is an arithmetic lattice, but is a finite index subgroup of $\Comm(p,\rho)$.
\par
\begin{defn}
Two period data $(p_i,\rho_i)$, $i=1,2$ are \textit{commensurable} if 
there exists $\Gamma'\subgp G$, which is a subgroup of finite index in $\Gamma_i = \dom(\rho_i)$ for 
$i=1,2$, such that $\rho_1 = \rho_2$ on $\Gamma'$. 
\end{defn}
\par
As is easily seen, commensurability is an equivalence relation. Note that by Proposition \ref{prop:uniqueness of period map}, $p_1 = p_2$ for commensurable period data.
%rho1 = rho2 auf Durschnitt stimmt nur modulo \pm I!
\par
There is a left action of $G\times G$ on modular embeddings. For $(g,h)\in G\times G$ 
and a modular embedding $(p,\rho)$,
\[g(p,\rho)h^{-1} := (g\circ p\circ h^{-1}, c_g\circ \rho \circ c_{h^{-1}}),\]
where $c_g: G\to G,\ \tilde g \mapsto g\tilde g g^{-1}$ is the action of an element $g\in G$ by conjugation.
\begin{defn}
We call two modular embeddings $(p_i,\rho_i)$, ($i=1,2$) \textit{weakly commensurable} if they become
commensurable under this action, \ie if there exist $(g,h)\in G\times G$ and $\Gamma'$, a subgroup of finite index in both
$\Gamma_1$ and $h\Gamma_2 h^{-1}$, such that $\rho_1 = c_g\circ \rho_2 \circ c_{h^{-1}}$ on $\Gamma'$. 
\end{defn}
\par
%  There is an
% action of $G$ from the left on $(p,\rho)$: for $g\in \SL_2(\RR)$, we
% send $(p,\rho)$ to $(g\circ p, c_g\circ \rho)$, where $c_g:h\mapsto ghg^{-1}$ is
% conjugation by $g$. There is also an action by $\SL_2(\RR)$ from the right on
% $(p,\rho)$: an element $h\in\SL_2(\RR)$ sends $(p,\rho)$ to $(p\circ h^{-1},
% \rho\circ c_{h^{-1}})$. 
\begin{exa}
 Clearly, the two modular embeddings from Section \ref{sec:examples} are not commensurable, since otherwise they would agree on $\gen{T^{2m}}$ for some $m\in \NN$, but $\rho_{\L22,2}(T^2)$ is parabolic whereas $\rho_{\Qmod9,4}(T^2)$ is elliptic.
\par 
 Moreover, for no two matrices $(g,h)\in \SL_2(\ZZ)^2$ is $g(p_{\L22,2},\rho_{\L22,2})h^{-1}$ commensurable with $(p_{\Qmod9,4},\rho_{\Qmod9,4})$, since conjugation by $h$ cannot exchange the cusps, as they are of different width, and conjugation by $g$ preserves the type (parabolic, respectively elliptic) of the image of a parabolic.
\par
 It remains to decide whether the two modular embeddings are not weakly commensurable, and more generally whether $(p_{\Qmod9,4},\rho_{\Qmod9,4})$ is (weakly) commensurable to a non-uniformizing representation of an arithmetic Teichm\"uller curve in $\Stratum_2(2)$ (\ie one generated by a square-tiled surface). However, we can exclude that $\rho_{\Qmod9,4}$ is weakly commensurable to a non-uniformizing 
representation of an arithmetic Teichm\"uller curve in $\Stratum_2(1,1)$. This is a consequence of the following discussion and the fact that such curves have non-negative Lyapunov spectrum $1,\tfrac1{2}$.
\end{exa}
\par
\begin{defn}
If $(p,\rho)$ is a cofinite modular embedding, we define its Lyapunov exponent to be
\[\lambda(p,\rho) = \frac{\deg(\ol{p})\vol(\HH/\range(\rho))}{\vol(\HH/\dom(\rho))}.\]
\end{defn}
\par
This definition is justified by Theorem~\ref{main thm} in that if $\rho$ admits a lift to $\SL_2(\RR)$ and if $\dom(\rho)$ acts freely, we obtain a VHS with Lyapunov exponent $\lambda(p,\rho)$.
\par
\begin{prop}
 The Lyapunov exponent of a modular embedding is a weak commensurability invariant.
\end{prop}
\begin{proof}
The value of $\lambda(p,\rho)$ clearly remains unchanged 
under the $G\times G$-action and under passage to a finite-index subgroup by Lemma \ref{lem:going up}.
\end{proof}
\par
\begin{defn}
For a modular embedding, we define the \textit{commensurator} 
\[\Comm(p,\rho) := \set{(g,h)\in G\times G}{(p,\rho), g(p,\rho)h^{-1}\ \text{are commensurable}}.\]
\end{defn}
\par
\begin{rem}
$\Comm(p,\rho)$ is a group containing $\Gamma =\dom(\rho)$ via $\gamma\mapsto (\gamma,\rho(\gamma))$. Moreover, for two
modular embeddings $(p_i,\rho_i)$, $i=1,2$ that are commensurable, we have $\Comm(p_1,\rho_1) = \Comm(p_2,\rho_2)$.
\end{rem}
\par
\begin{draftcomment}
If $(g,h)$, $(\tilde g,\tilde h)\in \Comm(p,\rho)$, then $g(p,\rho)h^{-1}$ 
and $\tilde g(p,\rho)\tilde h^{-1}$ are commensurable by transitivity. 
Thus, since acting by $(\tilde g^{-1},\tilde h^{-1})$ on the pair 
$g(p,\rho)h^{-1}, \tilde g(p,\rho)\tilde h^{-1}$ preserves commensurability,
we have that $\tilde g^{-1}g(p,\rho)h^{-1}\tilde h$, $(p,\rho)$ are commensurable 
and hence $(\tilde g^{-1}g,\tilde h^{-1} h)\in \Comm(p,\rho)$.
\par
Let $(g,h)\in \Comm(p_1,\rho_1)$. Then $(p_1,\rho_1)$ and $g^{-1}(p_1,\rho_1)h$ 
are commensurable, thus $(p_2,\rho_2)$, $g^{-1}(p_1,\rho_1)h$ are commensurable, 
thus $g(p_2,\rho_2)h^{-1}$, $(p_1,\rho_1)$ are commensurable, 
thus $g(p_2,\rho_2)h^{-1}$, $(p_2,\rho_2)$ are commensurable, thus $(g,h)\in \Comm(p_2,\rho_2)$.
\end{draftcomment}

 Further, $\Comm(p,\rho)$ maps into the commensurator 
\[\Comm(\dom(\rho)) = \set{h\in G}{h\dom(\rho)h^{-1}, \dom(\rho)\ \text{are commensurable}}\]
by $(g,h)\mapsto h$, and this map is injective if $p$ is not constant. If $p$ is not constant, we can therefore consider $\Comm(p,\rho)$ as a subgroup of $G$. In fact, it maps into $G_p = \set{h\in G}{\exists g\in G: g\circ p = p \circ h}$ by rigidity.
\par
As in the case of the usual commensurator, we have the following dichotomy. This is proved verbatim as in e.g. \cite[Prop. 6.2.3]{ZimmerErgodic84}.
\par
\begin{prop}
$\Comm(p,\rho)$ is either dense in $G$ or $\Gamma\subgp \Comm(p,\rho)$ is a subgroup of finite index.
\end{prop}
\par
\begin{draftcomment}
Let $H$ be the closure of $\Comm(p,\rho)$ in the Hausdorff topology of $G$, and let $H^0$ be the connected component of the identity. $\Gamma$ normalizes $H^0$ since $\Gamma\subgp H$. Then the Lie algebra $L(H^0)$ is a subspace of the Lie algebra $L(G)$ of $G$ invariant under $\Ad(\Gamma)$. Since $\Gamma$ is Zariski-dense in $G$, $L(H^0)$ is even invariant under $\Ad(G)$, hence $G$ normalizes $H^0$. Since $G$ is simple, $H^0 = \{e\}$ or $H^0 = G$. In the first case, $H^0$ is discrete, in the second $\Comm(p,\rho)$ is dense. 
\end{draftcomment}
\par
\begin{prop}\label{prop:comm of finite index}
 Suppose we are given a modular embedding $(p,\rho)$ such that $p$ is non-constant and $\rho$ has a nontrivial kernel. Then $\Gamma\subgp \Comm(p,\rho)$ is of finite index.
%can replace infinite kernel by nontrivial kernel since I am in the adjoint case
\end{prop}
\par
\begin{proof}
 Assume $\Comm(p,\rho)$ is dense in $G$. We claim that $G_p = G$. For let $h\in G$, and let $h_n\in \Comm(p,\rho)$ be a sequence such that $h_n\to h$ for $n\to \infty$ in the Hausdorff topology of $G$. For each $h_n$ there is $g_n\in G$ such that $g_n\circ p = p \circ h_n$. We claim that $(g_n)_n$ converges to $g\in G$. We show that $g_n$ is a Cauchy sequence, \ie for all $\varepsilon >0$ there is $n_0$ such that for all $n,m>n_0$ and $z\in \HH$, $d_{\HH}(g_nz,g_mz) < \varepsilon$. It suffices to show this only for all $z$ in some open subset of $\HH$, \eg in $p(\HH)$.
%Claim: if g_n(z) converges pointwise at 3 or more points then the limit exists and is a moebius
%transformation
Then  by the Schwarz-Pick lemma,
\[d_{\HH}(g_nz,g_mz) = d_{\HH}(g_np(w),g_mp(w)) \leq d_{\HH}(h_nw,h_mw) \to 0\]
uniformly in $w\in \HH$.
Thus $g_n\to g$, and $gp(z) = \lim_n g_n(p(z)) = \lim_n p(h_n(z)) = p(\lim_n h_n(z)) = p(hz)$ by continuity of $p$, and $h\in G_p$.
\par
If $G_p = G$, then $\rho$ admits an extension $\rho':G\to G$ by definition of $G_p$. But then $\ker(\rho')$ is a nontrivial, proper normal subgroup of $G$, contradicting the fact that $G$ is simple.
\end{proof}
\par
\begin{rem}
 In both our examples of Section \ref{sec:examples}, there is a nontrivial kernel. Thus we can apply Proposition \ref{prop:comm of finite index}. Since in both cases, $\deg(\ol{p}) = 1$ and the image group is $\SL_2(\ZZ)$, which is finitely maximal (\ie it is not properly contained in a bigger Fuchsian group), we find that the commensurator of $(\rho_{\L22,2},p_{\L22,2})$, respectively $(\rho_{\Qmod9,4}, p_{\Qmod9,4})$ coincides with $\Gamma_{\Theta}$.
\end{rem}

\bibliographystyle{amsalpha}
\bibliography{kappes_biblio}
\end{document}